\documentclass[12pt, twoside]{article}
\pdfoutput=1
\usepackage[margin=.75in]{geometry}

\usepackage{graphicx}
\usepackage{amsmath}
\usepackage{amssymb,amsfonts}
\usepackage{amsthm}
\usepackage{cite}
\usepackage{hyperref}
\usepackage{bm}
\usepackage{mathrsfs}
\usepackage{amssymb}
\usepackage{afterpage}
\usepackage{subfig}
\usepackage{xcolor,enumitem,enumerate}

\usepackage{hyperref}
\usepackage[nameinlink]{cleveref}

\newtheorem{theorem}{Theorem}[section]
\newtheorem{lemma}[theorem]{Lemma}

\newtheorem{corollary}[theorem]{Corollary}

\theoremstyle{definition}
\newtheorem{definition}{Definition}

\theoremstyle{remark}
\newtheorem{remark}[theorem]{Remark}

\usepackage{fancyhdr}
\title{A New Approach to Proper Orthogonal Decomposition with Difference Quotients}

\begin{document}
	
	\author{Sarah K.~Locke%
		\thanks{Department of Mathematics and Statistics, Missouri University of Science and Technology, Rolla, MO (\mbox{sld77@mst.edu}, \mbox{singlerj@mst.edu}).}
		\and
		John~R.~Singler%
		\footnotemark[1]
	}
	\maketitle
	
	\begin{abstract}
		In a recent work [B.\ Koc et al., arXiv:2010.03750, SIAM J.\ Numer.\ Anal., to appear], the authors showed that including difference quotients (DQs) is necessary in order to prove optimal pointwise in time error bounds for proper orthogonal decomposition (POD) reduced order models of the heat equation. In this work, we introduce a new approach to including DQs in the POD procedure. Instead of computing the POD modes using all of the snapshot data and DQs, we only use the first snapshot along with all of the DQs and special POD weights. We show that this approach retains all of the numerical analysis benefits of the standard POD DQ approach, while using a POD data set that has half the number of snapshots as the standard POD DQ approach, i.e., the new approach is more computationally efficient. We illustrate our theoretical results with numerical experiments.
	\end{abstract}
	
	
\section{Introduction}

A common model order reduction technique used for approximating partial differential equations (PDEs) and other mathematical models is proper orthogonal decomposition (POD). With this method, data from simulations or experiments are used to create modes, which are then used in a projection method to create a reduced order model (ROM). Because of the efficiency of POD ROMs, they are used in many applications including control theory and fluid dynamics to solve computationally demanding problems. For a small selection of applications see, e.g., \cite{Nguyen17, JinZhou17, ROWLEY2005, Bergmann05, Willcox02, Higham2020, Kunisch08, AllaFalconeVolkwein17,Rubino18, IliescuWang13, Graessle2019, Koc2019, Karasoezen2018}. 

Because of the widespread use of POD in applications, many researchers have studied POD ROMs from a numerical analysis perspective; see, e.g., \cite{Kunisch2001, KunischVolkwein02, LuoChenNavonYang08, IliescuWang14, IliescuWang14_DQ, Chapelle12, Rathinam2003,Singler14,LockeSingler20}. In order for POD to be beneficial for applications, researchers must understand how the various errors involved behave. To fully understand this for PDEs, three types of error must be considered: spatial discretization error, time discretization error, and ROM discretization error. The optimality of these errors is of particular concern. POD numerical analysis papers tend to focus on the time discretization error and the ROM discretization error since the spatial discretization error can typically be handled using existing techniques. For more information on these three types of error and the numerical analysis of POD, see the introduction of the recent work \cite{Koc2020}. The current work focuses only on the POD ROM errors and leads to an improved understanding of the numerical analysis of POD ROMs, particularly in regards to difference quotients and pointwise error bounds. 

There are multiple approaches to creating POD modes from the data. In this work, two of the most common existing methods are considered, and we introduce a new method. The first existing approach is a standard method to compute the POD modes and uses only the data, and the second existing approach utilizes both the data and the difference quotients (DQs) of the data. Researchers originally started including DQs in the POD calculations to improve the numerical analysis results for POD reduced order models as in \cite{Kunisch2001}. Further, DQs have been used in many applications including \cite{Hoemberg2003, Hijazi2020, Herkt2013, JinZhou17, KeanSchneier19, Kostova-VassilevskaOxberry18, Leibfritz2007, Sachs2013, ZhuDedeQuarteroni17}.  For years, researchers have been curious about the role DQs play in the behavior of the ROM and whether or not they should be included in the POD computations. In general, results on this topic were inconclusive. However in 2014, substantial progress was made by Iliescu and Wang in \cite{IliescuWang14_DQ} towards answering this question. In this work a notion of optimality was introduced and results strongly suggested that DQs were needed to achieve optimal pointwise-in-time convergence rates. Recently in \cite{Koc2020}, further progress was made. In this work it is shown that a critical assumption often made when using standard POD without DQs is automatically guaranteed to be satisfied when DQs are included with the data. The notion of optimality introduced in \cite{IliescuWang14_DQ} is extended, and it is shown that including difference quotients results in in pointwise POD projection error bounds and optimal pointwise ROM errors. The goal of the current work is to further investigate and understand pointwise error bounds in the POD-ROM setting.

To do this, we introduce a new approach to deriving POD modes from the data. When using all of the data with all of the DQs, the resulting data set is linearly dependent, i.e. the data set being used contains redundant information. This also leads to more costly POD basis computations compared to standard POD without DQs. In order to improve this situation, we consider the following question: \textit{Can we obtain all of the same numerical analysis benefits of using DQs with POD using a data set without redundancy?} We show that the answer is yes, if we choose the data set and POD weights in a correct way. For our new approach we use only the first data snapshot and all of the difference quotients. This new approach to using DQs with POD uses a data set without redundancy in the following sense: if the original set of $ M $ snapshots is linearly independent, then the data set used in our new DQ approach has dimension $ M $ and is also linearly independent. Using this new collection of data and special POD weights, we are not only able to approximate the DQs and the one regular snapshot, but all of the other regular snapshot data as well. With this method we also prove that we retain the numerical analysis benefits that come with having the DQs in the POD data set. 

The rest of the work is organized as follows. First we introduce the details for both current approaches to POD in \Cref{sec:standardPOD}. Then we describe our new computational approach in \Cref{sec:New_POD_DQ} and also prove POD data approximation error formulas, present preliminary computations, and prove pointwise POD data approximation error results. This is followed by reduced order model error analysis in \Cref{sec:ROM} and results from more detailed numerical computations in \Cref{sec:rom_computations}.

\section{Proper Orthogonal Decomposition}\label{sec:standardPOD}

In this section we introduce current approaches to POD: the standard POD approach and standard POD with difference quotients approach. We compare our new method to known results about these established methods throughout the work. For details on the basics of POD see, e.g., \cite{Djouadi08,GubischVolkwein17,HolmesLumleyBerkoozRowley12,KunischVolkwein02,Volkwein04,QuarteroniManzoniNegri16,Liang02}.

First we establish some general notation. Let $M$ and $N$ be a positive integers and $X$ and $Y$ be Hilbert spaces where the space $X$ is called the POD space. It is possible for $Y = X$. 
In the examples in this work for these Hilbert spaces we use the standard function spaces $L^2(\Omega)$ and $H_0^1(\Omega)$, where $\Omega$ is the spatial domain. Further, let $\mathbb{K} = \mathbb{R}$ or $\mathbb{K} = \mathbb{C}$. In order to consider variable weights that often arise with numerical integration, we define $S := \mathbb{K}_\Gamma^M$ with the weighted inner product given by
$$(g,h)_S = h^*\Gamma g = \sum_{j=1}^{M} \gamma_j g^j \overline{h^j}$$
where $g,h \in S$, $\Gamma = \text{diag}(\gamma_1, \gamma_2, \ldots , \gamma_M)$, and the values $\{\gamma_j\}_{j=1}^M$ are positive weights. In some instances it is beneficial to take the the positive weights to be certain specific values in order to approximate various time integrals. 

For the POD reduced order modeling in this work, we consider data sets consisting of approximate solution data for a time dependent partial differential equation. Throughout, we consider the time interval $[0,T]$ with $T > 0$ a fixed positive constant. The approximate solution data will be given at times $t_n = (n-1) \Delta t$, for $n = 1, \ldots, N$, where the time step is given by $\Delta t = \frac{T}{N-1}$. Note that while $T$ is fixed, $N$ can vary.

This work also uses projections. For a normed space $Z$, a bounded linear operator $\Pi:Z \to Z$ is a \textit{projection} onto $Z_r \subset Z$ if $\Pi^2 = \Pi$ and the range of $\Pi = Z_r$. Then $\Pi z = z$ for any $z \in Z_r$. Note that we do not assume that a projection is orthogonal unless explicitly stated. 

Further details and a selection of known results for the standard POD approach and the standard approach to POD including difference quotients can be found in \Cref{ssec:standardpod} and \Cref{ssec:POD_DQ} respectively. 

\subsection{Standard POD}\label{ssec:standardpod}
	
First we introduce the standard POD problem and operator. For this POD problem difference quotients are not considered. Let $W = \{w^j\}_{j=1}^M \subset X$ be the POD data, also called the POD snapshots, for some integer $M>0$. Given $r>0$, the standard POD problem is to find an orthonormal basis $\{\varphi_j\}_{j=1}^M \subset X$, also called the POD basis, minimizing the data approximation error 
\begin{equation}\label{eqn:data_approx_error_standardPOD}
E_r = \sum_{j=1}^M \gamma_j \| w^j - \Pi_r^X w^j \|_X^2 
\end{equation}
where $ \Pi_r^X : X \to X $ is the orthogonal projection onto $ X_r = \mathrm{span}\{ \varphi_k \}_{k=1}^r  $ defined by
\begin{equation}\label{Pidef}
\Pi_r^X x = \sum_{k=1}^r (x,\varphi_k)_X \varphi_k.
\end{equation}
The operator that provides the solution to this problem is $ K : S \to X $ which is given by
\begin{equation}\label{XPODdiscrete}
K f = \sum_{j=1}^M \gamma_j f^j \, w^j,  \quad  f = [ f^1, f^2, \ldots, f^M ]^T.
\end{equation}
We call this operator the standard POD operator. This operator $K$ is compact and has a singular value decomposition given by $ \{ \lambda_k^{1/2}, f_k, \varphi_k \} \subset \mathbb{R} \times S \times X $ where $
\{\lambda_k^{1/2}\}$ are the singular values and $\{f_k\}$ and $\{\varphi_k \}$ are the orthonormal singular vectors. The singular vectors $ \{ \varphi_k \} $ are called the POD modes of the data $ \{ w^k \} \subset X $ and $\{ \lambda_k^{1/2} \}$ are called the POD singular values. For more information about the singular value decomposition of compact operators, see, e.g., \cite[Chapters VI--VIII]{GohbergGoldbergKaashoek90}, \cite[Section V.2.3]{Kato95}, \cite[Chapter 30]{Lax02}, \cite[Sections VI.5--VI.6]{ReedSimon80}.

%For a positive integer $ r $, define $ X_r $. i.e., for $ x \in X $ fixed, $ x_r = \Pi_r^X x $ minimizes the approximation error $ \| x - x_r \|_X $ for $ x_r \in X_r $.  Also, since $ \{ \varphi_k \} $ is an orthonormal set in $ X $, we have the exact representation
%

%
The POD modes provide the best low rank approximation to the data, and it is known that
\begin{equation}\label{knownPODerror}
E_r = \sum_{j=1}^M \gamma_j \| w^j - \Pi_r^X w^j \|_X^2  =  \sum_{k = r+1}^s  \lambda_k
\end{equation}
where $\{\lambda_k\} $ are the POD eigenvalues and $s$ is the number of positive POD singular values. 
%The POD eigenvalues are the eigenvalues of $KK^*$ where $K^*: X \to S$ is the adjoint operator of $K:S \to X$ and is defined by 
%\begin{equation}
%K^*x = [(x,w_1)_X, (x,w_2)_X, \ldots  (x, w_s)_X]^T.
%\end{equation}
%
%When $ \lambda_k > 0 $, we have
%$$
%K f^k = \lambda_k^{1/2} \varphi_k.
%$$
%Thus we get the following formula for the POD modes
%\begin{equation}
%\varphi_k = \lambda_k^{-1/2} \sum_{j=1}^N \gamma_j (f^k)_j \, w^j
%\end{equation}
%which is used in numerical computations.
%\blue{add stuff about other standard POD details}

%The following lemma and proof are provided for completeness to show the error formulas for this POD problem using a result from our previous work \cite{LockeSingler20}.

For certain choices of weights $\{\gamma_i\}$ the error given in \Cref{eqn:data_approx_error_standardPOD} approximates a time integral, or a constant multiple of a time integral, as more and more time steps are used. Using various quadrature rules to determine the appropriate POD weights will lead to different time integral approximations. Allowing the weights to vary for the standard POD problem will also allow us to apply known results for this approach to new approaches.

The following lemma provides exact formulas for POD data approximation errors using other norms and other projections. We use this result later to provide POD data approximation results for other POD approaches. The proof of \Cref{POD_error} is similar to existing proofs of closely related results that exist in the literature so we omit the proof here; see,  e.g. \cite[Lemma 2.2]{Koc2020}, \cite[Theorem 5.1]{LockeSingler20}.  
\begin{lemma}\label{POD_error} 
	Let $ X_r = \mathrm{span}\{ \varphi_k \}_{k=1}^r $ and $ \Pi_r^X : X \to X $ be the orthogonal projection onto $ X_r $. Let $s$ be the number of positive POD singular values for $K$ defined in \Cref{XPODdiscrete}. If Y is a Hilbert space with $W \subset Y$ then
	\begin{equation}
	\sum_{j=1}^{M} \gamma_j \|w^j-\Pi_r^X w^j \|_Y^2  = \sum_{i=r+1}^{s} \lambda_i \|\varphi_i\|_Y^2.
	\end{equation}
	 In addition if $\pi_r : Y \to Y$ is a bounded linear projection onto $X_r$ then, 
	\begin{equation}
	\sum_{j=1}^{M} \gamma_j \|w^j-\pi_r w^j \|_Y^2  = \sum_{i=r+1}^{s} \lambda_i \|\varphi_i - \pi_r \varphi_i\|_Y^2.
	\end{equation}
\end{lemma}
%
  
%\begin{proof}
%	For this proof we will show that all conditions are met to satisfy Theorem 5.1 from \cite{LockeSingler20}. First, define $L:X \to Y$ such that $Lv = v$ for all $v \in W$. Then $\mathcal{D}(L)=\mathcal{R}(L) = W$. Also note that provided $\lambda_r >0$, we know $\varphi_k \in \mathcal{D}(L)$ and $X_r \subset \mathcal{D}(L) = W$. Together with the assumptions, we have $w^j \in \mathcal{D}(L)$, $w^j \in \mathcal{D}(\Pi_r^X)$ since $W\subset X$, and $w^j \in \mathcal{D}(\Pi_r^Y)$ since $W\subset Y$. We also know $\Pi_r^Y w^j \in \mathcal{D}(L^{-1})$ since $\mathcal{R}(\Pi_r^Y) \subset X_r \subset \mathcal{D}(L)$ and $L\varphi_k \in \mathcal{D}(\Pi_r^Y)$ since $\mathcal{R}(L) \subset W \subset Y =\mathcal{D}(\Pi_r^Y)$.	
%	Thus, all conditions are met for Theorem 5.1 and that completes the proof of this special case. 
%\end{proof}

The standard POD approach does not have general bounds for pointwise errors, as shown in \cite[Section 3]{Koc2020}.

\subsection{POD with Difference Quotients}\label{ssec:POD_DQ}

Another common approach to POD involves the use of difference quotients. Throughout this work, we refer to this method as the standard DQ approach. This approach has been studied by many including \cite{ HinzeVolkwein, Kunisch2001, Kostova-VassilevskaOxberry18, IliescuWang14_DQ, Koc2020, Gu2021}. In this work we consider backward Euler for the time stepping scheme and the difference quotients. 

Let $U = \{u^j\}_{j=1}^N \subset X$ be a given data set. Then the problem is to find an orthonormal basis minimizing the error
\begin{equation}\label{POD_DQ_problem}
E_r^{DQ} = \sum_{j=1}^{N} \Delta t \|u^j-\Pi_r^X u^j \|_X^2 + \sum_{j=1}^{N-1} \Delta t \|\partial u^j - \Pi_r^X \partial u^j \|_X^2
\end{equation} where $\Delta t$ is the time step and the difference quotients are given by 
\begin{equation}\label{diffquotient}
\partial u^j = \frac{u^{j+1}- u^j}{\Delta t}.
\end{equation}
These difference quotients approximate the time derivative of the data in continuous time. The operator that provides the minimizing basis for the error in \Cref{POD_DQ_problem} is $K_{DQ} : S \to X$ defined by 
\begin{equation}\label{POD_DQ_operator}
K_{DQ} f = \sum_{j=1}^{N} \Delta t f^{j} u^j + \sum_{j=1}^{N-1} \Delta t f^{N+j} \partial u^j.
\end{equation}

This approach uses a total of $M = 2N-1$ data snapshots which is nearly twice as many as the standard POD approach. Note for this operator we have $K_{DQ} f = Kf $ where $w^i = u^i$ and $\gamma_i = \Delta t$ for $i = 1, \ldots  N$, and also $w^{N+i} = \partial u^i$ and $\gamma_{N+i} = \Delta t$ for $i = 1,\ldots ,N-1$. The resulting POD data set is $\{w^j\}_{j=1}^{M}$, where $M = 2N-1$. Taking $\{\lambda_j^{DQ}\}_{j=1}^{2N-1}$ to be the POD eigenvalues and keeping the same notation $\{\varphi_j\}$ for the POD basis functions, we get similar results to those for the standard POD operator. When using this new set $\{w^j\}$ as the POD data set, we not only have a set that is nearly twice as large as the original but it can also be checked that the new set is linearly dependent. This redundancy is something we avoid with our new approach introduced in \Cref{sec:New_POD_DQ}.
%the following formula for the POD modes
%\begin{equation*}
%\varphi_k = (\lambda^{DQ}_k)^{-1/2} K_{DQ} f^k
%\end{equation*}
%and 

\begin{remark}
	While the weights can be taken to be any set of positive constants, for simplicity in this work we take them to be the constant $\Delta t$. The results in this work can be extended to variable weights or other choices of constant weights. One popular choice of constant weight in the literature is $M^{-1}$ as in \cite{Koc2020,IliescuWang14_DQ} where $M$ represents the total number of data snapshots for the standard difference quotient approach to POD. Similarly to the standard POD case, with certain choices of weights, one can approximate time integrals with various quadrature rules.
\end{remark}

%\blue{compare to standard approach. including how the POD modes (formula) compare, etc} 

%\blue{add in all the other POD details as well}

The following result is also similar to Lemma 2.4 in \cite{Koc2020}. We provide it here for completeness. 

\begin{lemma}\label{POD_error_DQ} 
	Let $ X_r = \mathrm{span}\{ \varphi_k \}_{k=1}^r $ and $ \Pi_r^X : X \to X $ be the orthogonal projection onto $ X_r $. Let $s$ be the number of positive POD singular values for $K_{DQ}$ defined in \Cref{POD_DQ_operator}. We have the following error formula:
	\begin{equation}
	\sum_{j=1}^{N} \Delta t \|u^j-\Pi_r^X u^j \|_X^2 + \sum_{j=1}^{N-1} \Delta t \|\partial u^j - \Pi_r^X \partial u^j \|_X^2 = \sum_{i=r+1}^{s} \lambda_i^{DQ}.
	\end{equation}
	If Y is a Hilbert space with $U \subset Y$ then
	\begin{equation}
	\sum_{j=1}^{N} \Delta t \|u^j-\Pi_r^X u^j \|_Y^2 + \sum_{j=1}^{N-1} \Delta t \|\partial u^j - \Pi_r^X \partial u^j \|_Y^2  = \sum_{i=r+1}^{s} \lambda_i^{DQ} \|\varphi_i\|_Y^2.
	\end{equation}
	In addition, if $\pi_r : Y \to Y$ is a bounded linear projection onto $X_r$ then
%	\begin{equation}
%	\sum_{j=1}^{N} \Delta t \|u^j-\Pi_r^Y u^j \|_X^2 + \sum_{j=1}^{N-1} \Delta t \|\partial u^j - \Pi_r^Y \partial u^j \|_X^2  = \sum_{i=r+1}^{s} \lambda_i^{DQ} \|\varphi_i - \Pi_r^Y \varphi_i\|_X^2.
%	\end{equation}
%	and
	\begin{equation}
	\sum_{j=1}^{N} \Delta t \|u^j-\pi_r u^j \|_Y^2 + \sum_{j=1}^{N-1} \Delta t \|\partial u^j - \pi_r \partial u^j \|_Y^2  = \sum_{i=r+1}^{s} \lambda_i^{DQ} \|\varphi_i - \pi_r \varphi_i\|_Y^2.
	\end{equation}
\end{lemma}
\begin{proof}
	This result follows from \Cref{knownPODerror} and \Cref{POD_error} by taking $\{\lambda_j^{DQ} \}_{j=1}^s$ as the POD eigenvalues for the POD operator in \Cref{POD_DQ_operator}.
\end{proof}
We have the following result about the pointwise error bounds for this POD case. A similar result with a different set of constant weights can be found in \cite[Theorem 3.7]{Koc2020}. The proof of the result below is similar so we omit it here. That theorem was key to obtaining the optimal pointwise POD ROM error bounds which were a main contribution of that work.
%These pointwise error bounds are used to show POD ROM error bounds for the standard DQ case and are included here for completeness.
\begin{theorem}\label{thm:POD_DQ_pointwise_error}
	Let $ X_r = \mathrm{span}\{ \varphi_k \}_{k=1}^r $ and $ \Pi_r^X : X \to X $ be the orthogonal projection onto $ X_r $. Let $s$ be the number of positive POD singular values for $K_{DQ}$. We have 
	\begin{equation}
	\max_{1 \leq j \leq N} \|u^j - \Pi_r^X u^j \|_X^2 \leq C\left( \sum_{i=r+1}^s \lambda_i^{DQ} \right).
	\end{equation}
	If Y is a Hilbert space with $U \subset Y$ then
	\begin{equation}
	\max_{1 \leq j \leq N} \|u^j - \Pi_r^X u^j \|_Y^2 \leq C\left( \sum_{i=r+1}^s \lambda_i^{DQ} \|\varphi_i\|^2_Y \right).
	\end{equation}
%	\begin{equation}
%	\|u^j - \Pi_r^Y u^j \|_X^2 \leq C\left( \sum_{i=r+1}^s \lambda_i^{DQ} \|\varphi_i - \Pi_r^Y \varphi_i \|^2_X \right),
%	\end{equation}
%	and
	If $\pi_r : Y \to Y$ is a bounded linear projection onto $X_r$ then
	\begin{equation}
	\max_{1 \leq j \leq N} \|u^j - \pi_r u^j \|_Y^2 \leq C\left( \sum_{i=r+1}^s \lambda_i^{DQ} \|\varphi_i - \pi_r \varphi_i \|^2_Y \right)
	\end{equation}
	where $C = 2\max\{T^{-1}, T\} $.
\end{theorem}

In \Cref{sec:PointwiseErrorBounds} we obtain a similar pointwise POD projection error result for the new POD approach described below in \Cref{sec:New_POD_DQ}.

\section{A New Approach to POD with Difference Quotients}\label{sec:New_POD_DQ}

Next, we return to the question posed in the introduction: ``Can we obtain all of the same numerical analysis benefits of using DQs with POD using a data set without redundancy?" We obtain a positive answer to this question by introducing a new POD problem and operator. Instead of including all of the POD data snapshots and all of the difference quotients as in \Cref{ssec:POD_DQ}, we include the first data snapshot and all of the difference quotients. Thus for the data $U = \{u^j\}_{j=1}^N \subset X$, the new POD problem is to minimize the error given by  
\begin{equation}\label{POD_DQ1_problem}
E_r^{DQ1} = \|u^1-\Pi_r^X u^1 \|_X^2 + \sum_{j=1}^{N-1} \Delta t \|\partial u^j - \Pi_r^X \partial u^j \|_X^2
\end{equation}
where the difference quotients are defined by
%\begin{equation}
%\partial u^j = \frac{u^{j+1}- u^j}{\Delta t}
%\end{equation}
\Cref{diffquotient}. Note that the POD error function in \Cref{POD_DQ1_problem} does not include the weighted sum of the errors of the regular snapshots; this contrasts with the POD approaches in \Cref{sec:standardPOD}, which both include such error terms, see \Cref{eqn:data_approx_error_standardPOD} and  \Cref{POD_DQ_problem}. Furthermore, in the POD approaches in \Cref{sec:standardPOD}, we have exact error formulas for these error terms; see \Cref{POD_error} and \Cref{POD_error_DQ}. In \Cref{sec:PointwiseErrorBounds}, we consider the weighted sum of the errors of the regular snapshots for this new approach and obtain an approximation error result in \Cref{cor:totalbounds}.

Note that with this approach we have a total of $N-1$ difference quotients. Together with the single snapshot, we have a total of $N$ data snapshots for this POD problem. This is an improvement from the standard DQ approach to POD which has $2N-1$ data snapshots.

The minimum error can be found using the POD operator:
\begin{equation}\label{POD_DQ1_operator}
K_1 f = f^1 u^1 + \sum_{j=1}^{N-1} \Delta t f^{j+1} \partial u^j
\end{equation}
We have $K_1 f = K f$ where $w^1 = u^1$, $\gamma_1= 1$, $\gamma_{i+1} = \Delta t$, and $w^{i+1} = \partial u^i$ for $i = 1,\ldots ,N-1$. We choose to use the constant time step, $\Delta t$, as the weight for simplicity throughout the work. The results can be extended to include variable weights as well. Recall that for certain choices of weights one can approximate a time integral and the difference quotients approximate a time derivative.
%If $\lambda_k^{DQ1} >0$ then 
%\begin{equation*}
%\varphi_k = (\lambda^{DQ1}_k)^{-1/2} K_{1} f^k
%\end{equation*}

Note that in contrast to the data set created for the standard POD approach with DQs, the data set for this new approach is linearly independent if the original data set is linearly independent as shown in \Cref{DQ1_lin_indep}.

\begin{lemma}\label{DQ1_lin_indep}
	If $\{u^i\}_{i=1}^N$ is linearly independent then $\{w^i\}_{i=1}^{N}$ given by $w^1 = u^1$ and $w^{i+1} = \partial u^i$ for $i = 1,\ldots ,N-1$ is linearly independent. 
\end{lemma}

\begin{proof}
	We show that if 
	\begin{equation*}
	c_1 w^1 + c_2 w^2 + \cdots + c_N w^N = 0
	\end{equation*}
	then $c_i = 0$ for all $i=1, \ldots , N$. We have 
	\begin{equation*}
	c_1 u^1 + c_2 \left(\frac{u^2 - u^1}{\Delta t}\right) + \cdots + c_i \left(\frac{u^i-u^{i-1}}{\Delta t}\right) + \cdots + c_N \left(\frac{u^N-u^{N-1}}{\Delta t} \right) = 0
	\end{equation*} 
	then
	\begin{equation*}
	\left(c_1- \frac{c_2}{\Delta t}\right)u^1 + \left(\frac{c_2}{\Delta t} - \frac{c_3}{\Delta t}\right)u^2 +\cdots + \left(\frac{c_{N-1}}{\Delta t}- \frac{c_N}{\Delta t}\right)u^{N-1} + \frac{c_N}{\Delta t} u^N = 0.
	\end{equation*}
	Since $\{u_i\}$ is linearly independent we know each of these coefficients must equal $0$. Solving that system of equations leads to the conclusion that $c_i = 0$ for all $i = 1, \ldots , N$.
\end{proof}

If we let $\{\lambda_j^{DQ1} \}_{j=1}^N$ be the POD eigenvalues for this new POD approach, and let $\{\varphi_k\}_{k=1}^{r}$ be the POD modes for this data, we get the following error formulas given in \Cref{lem:POD_DQ1_error}. 

\begin{lemma}\label{lem:POD_DQ1_error} 
	Let $ X_r = \mathrm{span}\{ \varphi_k \}_{k=1}^r $ and $ \Pi_r^X : X \to X $ be the orthogonal projection onto $ X_r $. Let $s$ be the number of positive POD singular values for $K_1$ defined in \Cref{POD_DQ1_operator}. We have the following formula for the data approximation error: 
	\begin{equation}
	\|u^1-\Pi_r^X u^1 \|_X^2 + \sum_{j=1}^{N-1} \Delta t \|\partial u^j - \Pi_r^X \partial u^j \|_X^2 = \sum_{i=r+1}^{s} \lambda_i^{DQ1}.
	\end{equation}
	If Y is a Hilbert space with $U \subset Y$ then 
	\begin{equation}\label{equ:X_Ytotalerror}
	\|u^1-\Pi_r^X u^1 \|_Y^2 + \sum_{j=1}^{N-1} \Delta t \|\partial u^j - \Pi_r^X \partial u^j \|_Y^2 = \sum_{i=r+1}^{s} \lambda_i^{DQ1} \|\varphi_i \|^2_Y.
	\end{equation}
%	\begin{equation}
%	\|u^1-\Pi_r^Y u^1 \|_X^2 + \sum_{j=1}^{N-1} \Delta t \|\partial u^j - \Pi_r^Y \partial u^j \|_X^2 = \sum_{i=r+1}^{s} \lambda_i^{DQ1} \|\varphi_i - \Pi_r^Y \varphi_i\|^2_X,
%	\end{equation}
%	and 
	If in addition $\pi_r : Y \to Y$ is a bounded linear projection onto $X_r$ then,
	\begin{equation}
	\|u^1-\pi_r u^1 \|_Y^2 + \sum_{j=1}^{N-1} \Delta t \|\partial u^j - \pi_r \partial u^j \|_Y^2 = \sum_{i=r+1}^{s} \lambda_i^{DQ1} \|\varphi_i - \pi_r \varphi_i\|^2_Y.
	\end{equation}
\end{lemma}

\begin{proof}
	This result follows from \Cref{knownPODerror} and  \Cref{POD_error} by taking $\{\lambda_j^{DQ1} \}_{j=1}^s$ as the POD eigenvalues for the POD operator in \Cref{POD_DQ1_operator}.
\end{proof}

Preliminary computational results and pointwise error estimates for this new POD approach are discussed in \Cref{sec:some_comp_1} and \Cref{sec:PointwiseErrorBounds} respectively. 

\subsection{Preliminary Computations}\label{sec:some_comp_1}

Before moving to our main results we perform some preliminary computations to test the new POD approach with difference quotients. For all computations throughout this work we consider the following test problem.

\medskip

\noindent\textbf{Test Problem:} Consider the one dimensional heat equation 
	\begin{align*}
	u_t - \nu u_{xx} &= 0 , \text{ in } \Omega \times [0,T] \\
	u(x,0)&= e^{x} \sin(\pi x) 
	\end{align*}
	with $\nu=1$, $\Omega = [0,1]$, $T = 1$, and zero Dirichlet boundary conditions.
	
\medskip

For the data $\{u^j\}$, we compute the solution using the finite element method with linear elements, equally spaced nodes, and backward Euler with a constant time step for the time stepping. The initial condition is taken to be the linear interpolation of the initial condition with respect to the finite element nodes. For this data and with the POD space $X = L^2(\Omega)$, we can calculate the POD modes, POD singular values, and the data approximation errors exactly which allows for comparison between the errors in the formulas from \Cref{lem:POD_DQ1_error} and the actual approximation errors. In order to compute the singular value decomposition (SVD) of the POD operator, we use the technique described in \cite[Section 2.2]{Fareed18} with a minor modification to account for the POD weights. This procedure works well and is highly accurate for smaller data sets; for larger data sets one could use an incremental SVD approach or another related algorithm instead, see e.g. \cite{Baker2012, Brand2006, Himpe2018, Fareed2020, Fareed18} and the references therein. 

First we plot the singular values for both the standard POD and our new POD approach that includes one data snapshot and all of the difference quotients. For the standard POD computations we use the data set consisting of only the regular snapshots $\{ w^j \}_{j=1}^N = \{ u^j \}_{j=1}^N$ and we choose constant weights $\gamma_j = \Delta t$ for $j = 1, \ldots, N$. The singular value plots allow for a quick comparison between the two approaches to POD and are given in \Cref{fig:StandardandDQ1PODsvplot}. For each plot, we show the first $20$ POD singular values for $20, 50, 100,$ and $150$ finite element nodes when using $100$ equally spaced time steps. The POD singular values decay at a similar rate which indicates that the POD basis for each case has a similar ability to approximate the data.  

\begin{figure}%
	\centering
	\subfloat[Standard POD]{{\includegraphics[width=7cm]{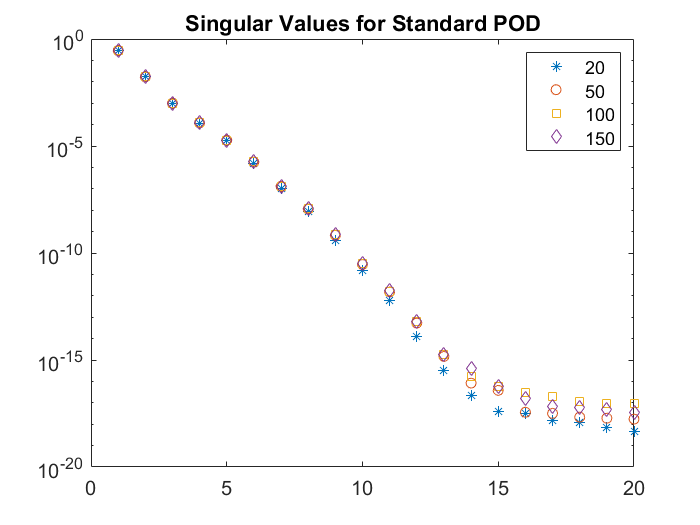} }}%
	\qquad
	\subfloat[New DQ POD]{{\includegraphics[width=7cm]{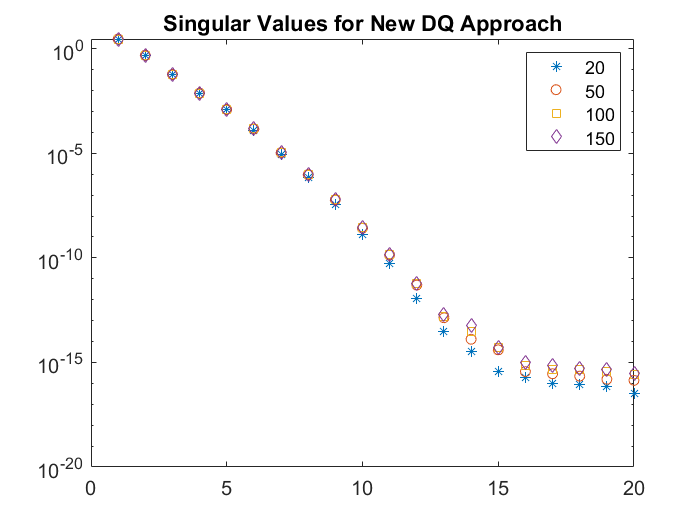} }}%
	\caption{Plot of singular values of POD operators using different numbers of finite element nodes}
	\label{fig:StandardandDQ1PODsvplot}
\end{figure}

Next, we consider the data approximation error results given in \Cref{lem:POD_DQ1_error} numerically. For these computations we take $200$ equally spaced time steps and $100$ equally spaced finite element nodes and compute the actual errors and the error formulas. Recall that $X = L^2(\Omega)$ is the POD space and note that here we take either $Y = H_0^1(\Omega)$ or $Y = L^2(\Omega)$. Also note the projection $\pi_r$ is taken to be the Ritz projection, which we discuss in more detail in \Cref{sec:ROM}. The results are shown for all formulas of \Cref{lem:POD_DQ1_error} in \Cref{tab:comp} for $r = 4$, $r=6$, and $r = 8$. For example, the second row for each $r$ value in the table gives the values for
$$
\text{actual error } = 	\|u^1-\Pi_r^X u^1 \|_{H_0^1}^2 + \sum_{j=1}^{N-1} \Delta t \|\partial u^j - \Pi_r^X \partial u^j \|_{H_0^1}^2 ,  
$$
and
$$
\text{error formula } = \sum_{i = r+1}^s \lambda_i^{DQ1} \|\varphi_k\|^2_{H_0^1}
$$
with the respective values for $r$. Round-off errors in the POD computations can cause very small imaginary parts to occur in the error formulas. Thus we report the absolute value of the computed error formulas. 
Note that the difference between the actual errors and the error formulas is unnoticeable to the given significant digits. These computational results verify what we show analytically. Similar results are achieved when using $X = H_0^1(\Omega)$. 

\begin{table}
	\renewcommand{\arraystretch}{1.25}
	\begin{center}	
		\begin{tabular}{|c|c|c|c|c|}
			\hline
			$r$ value & Projection & Norm & Actual Error & Error Formula \\
			\hline
			4 & $\Pi_r^X$ & $L^2(\Omega)$ & 9.761e-06 & 9.761e-06 \\
			&$\Pi_r^X$ & $H_0^1(\Omega)$ & 6.200e-03 & 6.200e-03 \\
			&$\pi_r$ & $L^2(\Omega)$ & 1.736e-05 & 1.736e-05 \\
			&$\pi_r$ & $H_0^1(\Omega)$ & 3.100e-03 & 3.100e-03 \\
			\hline
			6 & $\Pi_r^X$ & $L^2(\Omega)$ & 5.482e-09 & 5.482e-09 \\
			&$\Pi_r^X$ & $H_0^1(\Omega)$ & 1.036e-05 & 1.036e-05 \\
			&$\pi_r$ & $L^2(\Omega)$ & 1.133e-08 & 1.133e-08 \\
			&$\pi_r$ & $H_0^1(\Omega)$ & 5.159e-06 & 5.159e-06 \\
			\hline
			8 & $\Pi_r^X$ & $L^2(\Omega)$ & 1.557e-12 & 1.557e-12 \\
			&$\Pi_r^X$ & $H_0^1(\Omega)$ & 4.772e-09 & 4.772e-09 \\
			&$\pi_r$ & $L^2(\Omega)$ & 3.481e-12 & 3.481e-12 \\
			&$\pi_r$ & $H_0^1(\Omega)$ & 1.659e-09 & 1.659e-09 \\
			\hline
		\end{tabular}
		\caption{Actual Error vs. Error Formulas from \Cref{lem:POD_DQ1_error} \label{tab:comp} for New DQ POD}
	\end{center}
\end{table}

%\begin{table}
%	\begin{center}	
%		\begin{tabular}{|c|c|c|c|}
%			\hline
%			Projection & Norm & Actual Error & Error Formula \\
%			\hline
%			$\Pi_r^X$ & X & 9.761e-06 & 9.761e-06 \\
%			$\Pi_r^X$ & Y & 6.200e-03 & 6.200e-03 \\
%			$\Pi_r^Y$ & X & 1.736e-05 & 1.736e-05 \\
%			$\Pi_r^Y$ & Y & 3.100e-03 & 3.100e-03 \\
%			\hline
%		\end{tabular}
%		\caption{Actual Error vs. Error Formulas for $r = 4$\label{tab:comp_r4} for New DQ POD}
%	\end{center}
%\end{table}

%\begin{table}
%	\begin{center}	
%		\begin{tabular}{|c|c|c|c|}
%			\hline
%			Projection & Norm & Actual Error & Error Formula \\
%			\hline
%			$\Pi_r^X$ & X & 5.482e-09 & 5.482e-09 \\
%			$\Pi_r^X$ & Y & 1.036e-05 & 1.036e-05 \\
%			$\Pi_r^Y$ & X & 1.133e-08 & 1.133e-08 \\
%			$\Pi_r^Y$ & Y & 5.159e-06 & 5.159e-06 \\
%			\hline
%		\end{tabular}
%		\caption{Actual Error vs. Error Formulas for $r = 6$\label{tab:comp_r6} for New DQ POD}
%	\end{center}
%\end{table}

%\begin{table}
%	\begin{center}	
%		\begin{tabular}{|c|c|c|c|}
%			\hline
%			Projection & Norm & Actual Error & Error Formula \\
%			\hline
%			$\Pi_r^X$ & X & 1.557e-12 & 1.557e-12 \\
%			$\Pi_r^X$ & Y & 4.772e-09 & 4.772e-09 \\
%			$\Pi_r^Y$ & X & 3.481e-12 & 3.481e-12 \\
%			$\Pi_r^Y$ & Y & 1.659e-09 & 1.659e-09 \\
%			\hline
%		\end{tabular}
%		\caption{Actual Error vs. Error Formulas for $r = 8$\label{tab:comp_r8} for New DQ POD}
%	\end{center}
%\end{table}

Computational comparisons between all three of the methods can be found in \Cref{sec:rom_computations_comp} where we consider the errors in the reduced order models.

\subsection{Pointwise Error Bounds}\label{sec:PointwiseErrorBounds}

Using the technique from \cite[Lemma 3.6]{Koc2020} we establish the following lemma which allows us to directly prove the POD pointwise projection error bounds for the new DQ POD approach and an approximation error result for the weighted sum of the errors of the regular snapshots. These results are necessary to prove the reduced order model error bounds and show their optimality in \Cref{sec:ROM}. 

\begin{lemma}\label{lemma}
	Let $T>0$, $Z$ be a normed space, $\{z^j \}_{j=1}^{N} \subset Z$, and $\Delta t = T/(N-1)$. Then
	\begin{equation}
	\max_{1\leq j\leq N} \|z^j \|_Z^2 \leq C\left( \|z^1\|_Z^2 + \sum_{\ell=1}^{N-1} \Delta t \|\partial z^\ell \|_Z^2 \right)
	\end{equation}
	where $\partial z^\ell = \frac{z^{\ell+1}-z^\ell}{\Delta t}$ for $\ell = 1,\ldots ,N-1$, and $C = 2\max \{T,1\}$.
\end{lemma}

\begin{proof}
	Note using $z^j = z^1 + \sum_{\ell =1}^{j-1} \Delta t (\partial z^\ell)$, we have
	\begin{equation*}
	\|z^j\| \leq \|z^1\| + \sum_{\ell =1}^{j-1} \Delta t \| \partial z^\ell \| \leq \|z^1\| + \left( \sum_{\ell =1}^{j-1} \Delta t \right)^{1/2} \left(  \sum_{\ell =1}^{N-1} \Delta t \| \partial z^\ell \|^2 \right)^{1/2}.		
	\end{equation*}
	Then
	\begin{align*}
	\|z^j\|^2 &\leq 2\|z^1\|^2 + 2\left( \sum_{\ell =1}^{j-1} \Delta t \right) \left(  \sum_{\ell =1}^{N-1} \Delta t \| \partial z^\ell \|^2 \right) \leq 2\|z^1\|^2 + 2 T \sum_{\ell =1}^{N-1} \Delta t \| \partial z^\ell \|^2 	
	%& \leq 2\|z^1\|^2 + 2t_j \sum_{\ell =1}^{s-1} \Delta t \| \partial z^\ell \|^2 \\
	\end{align*}
	since $T = \sum_{j=1}^{N-1} \Delta t = (N-1) \Delta t$. Take the maximum over all $j$ and the result follows. 
\end{proof}

%This lemma allows us to show the following theorem and corollary about pointwise error bounds and total error bounds, respectively.
Next, we obtain a  pointwise POD projection error result for the new POD DQ approach that is very similar to \Cref{thm:POD_DQ_pointwise_error} for the standard POD DQ case.

\begin{theorem}\label{thm:POD_DQ1_pointwise_error}
	Let $ X_r = \mathrm{span}\{ \varphi_k \}_{k=1}^r $ and $ \Pi_r^X : X \to X $ be the orthogonal projection onto $ X_r $. Let $s$ be the number of positive POD eigenvalues for $U = \{u^j\}_{j=1}^N$. Then
	\begin{equation}
	\max_{1\leq j\leq N} \|u^j - \Pi_r^X u^j \|_X^2 \leq C\left( \sum_{i=r+1}^s \lambda_i^{DQ1} \right).
	\end{equation}
	If Y is a Hilbert space with $U \subset Y$ then
	\begin{equation}\label{eqn:X_Ypointwiseerror}
	\max_{1\leq j\leq N} \|u^j - \Pi_r^X u^j \|_Y^2 \leq C\left( \sum_{i=r+1}^s \lambda_i^{DQ1} \|\varphi_i\|^2_Y \right),
	\end{equation}
%	\begin{equation}
%	\max_{1\leq j\leq N} \|u^j - \Pi_r^Y u^j \|_X^2 \leq C\left( \sum_{i=r+1}^s \lambda_i^{DQ1} \|\varphi_i - \Pi_r^Y \varphi_i \|^2_X \right),
%	\end{equation}
%	and
	and in addition if $\pi_r : Y \to Y$ is a bounded linear projection onto $X_r$ then
	\begin{equation}
	\max_{1\leq j\leq N} \|u^j - \pi_r u^j \|_Y^2 \leq C\left( \sum_{i=r+1}^s \lambda_i^{DQ1} \|\varphi_i - \pi_r \varphi_i \|^2_Y \right)
	\end{equation}
	where $C =  2\max \{T,1\}$.
\end{theorem}

\begin{proof} To prove this theorem take $z^j = u^j-\Pi_r^X u^j$ or $z^j = u^j-\pi_r u^j$ with $Z = X$ or $Z = Y$ in \Cref{lemma}. For example, if we let $z^j = u^j-\Pi_r^X u^j$ and $Z = X$, then
	\begin{equation}
	\max_{1\leq j\leq N} \|u^j-\Pi_r^X u^j \|_X^2 \leq C\left( \|u^1-\Pi_r^X u^1\|_X^2 + \sum_{\ell=1}^{N-1} \Delta t \|\partial u^\ell-\Pi_r^X \partial u^\ell \|_X^2 \right)
	\end{equation}
	Applying \Cref{lem:POD_DQ1_error} for each of the three cases gives the result. 
\end{proof}

Next, we use the above pointwise error bounds to obtain error bounds for the weighted sum of the errors of the regular snapshots.

\begin{corollary}\label{cor:totalbounds}
		Let $ X_r = \mathrm{span}\{ \varphi_k \}_{k=1}^r $ and $ \Pi_r^X : X \to X $ be the orthogonal projection onto $ X_r $. Let $s$ be the number of positive POD eigenvalues for $U$ then
		\begin{equation}\label{eqn:cor_1}
		\sum_{j=1}^N \Delta t \|u^j - \Pi_r^X u^j \|_X^2 \leq C\left( \sum_{i=r+1}^s \lambda_i^{DQ1} \right).
		\end{equation}
		If Y is a Hilbert space with $U \subset Y$ then
		\begin{equation}\label{eqn:cor_2}
		\sum_{j=1}^N \Delta t\|u^j - \Pi_r^X u^j \|_Y^2 \leq C\left( \sum_{i=r+1}^s \lambda_i^{DQ1} \|\varphi_i\|^2_Y \right).
		\end{equation}
%		\begin{equation}
%		\sum_{j=1}^N \Delta t\|u^j - \Pi_r^Y u^j \|_X^2 \leq C\left( \sum_{i=r+1}^s \lambda_i^{DQ1} \|\varphi_i - \Pi_r^Y \varphi_i \|^2_X \right),
%		\end{equation}
%		and
		If in addition $\pi_r : Y \to Y$ is a bounded linear projection onto $X_r$ then
		\begin{equation}\label{eqn:cor_3}
		\sum_{j=1}^N \Delta t\|u^j - \pi_r u^j \|_Y^2 \leq C\left( \sum_{i=r+1}^s \lambda_i^{DQ1} \|\varphi_i - \pi_r \varphi_i \|^2_Y \right)
		\end{equation}
		where $C =  4\max \{T^2,T\}$.
\end{corollary}

\begin{proof}
	%Since $ T = \sum_{j=1}^N \Delta t = N \Delta t $, 
	Since $$N \Delta t = \frac{N}{N-1} [(N-1)\Delta t] = [N/(N-1)] T \leq 2T,$$ we have
	\begin{align*}
	\sum_{j=1}^N \Delta t \| u^j - \Pi_r^X u^j \|_X^2  %&\leq  N \Delta t \max_{1\leq j \leq N} \| u^j - \Pi_r^X u^j \|_X^2 \\
	%& = \frac{N}{N-1} (N-1) \Delta t \max_{1\leq j \leq N} \| u^j - \Pi_r^X u^j \|_X^2 \\
	& \leq 2T \max_{1\leq j \leq N} \| u^j - \Pi_r^X u^j \|_X^2.
	\end{align*}
	\Cref{thm:POD_DQ1_pointwise_error} gives \Cref{eqn:cor_1}. The proofs of \Cref{eqn:cor_2,eqn:cor_3} follow in the same way.	
\end{proof}

These results are similar to those for the standard DQ approach while keeping redundancy out of the data set.

\section{Reduced Order Modeling}\label{sec:ROM}
In this section we establish theory to compare the reduced order model solution to the backward Euler finite element solution for the heat equation using our new POD with DQs approach. In this section all POD computations are done using the new approach and all function spaces are assumed to be real. Our analysis and proof techniques strongly rely on the approach in \cite[Section 4]{Koc2020}. We provide proofs to make the work self-contained. 

Let $ \Omega \subset \mathbb{R}^d $ with $ d \geq 1 $ be an open bounded domain with Lipschitz continuous boundary, and define $ V = H^1_0(\Omega) $. The space $ V $ is a Hilbert space with inner product $(g,h)_{H_0^1} = (\nabla g , \nabla h )_{L^2}$. We consider the weak formulation of the heat equation with homogeneous Dirichlet boundary conditions:
\begin{equation}\label{weakformheateq}
(\partial_t u, v)_{L^2} + \nu(\nabla u , \nabla v )_{L^2} = (f,v)_{L^2} \quad \forall v \in V
\end{equation}
where $u(\cdot,0) = u^1$ is the initial condition, $\nu$ is a positive constant, and $f$ is a given forcing function. We project \Cref{weakformheateq} onto a standard conforming  finite element space $V^h \subset V$ and apply backward Euler to obtain
\begin{equation}\label{BEheateq}
\left( \frac{u^{n+1}-u^n}{\Delta t}, v \right)_{L^2} + \nu (\nabla u^{n+1} , \nabla v )_{L^2} = (f^{n+1} , v)_{L^2} \quad \forall v \in V^h.
\end{equation}

%Note $u$ represents the exact solution to the heat equation. 

%\begin{equation}\label{BEheateq_r}
%\left( \frac{u^{n+1}-u^n}{\Delta t}, v_r \right)_{L^2} + \nu (\nabla u^{n+1} , \nabla v_r )_{L^2} = (f^{n+1} , v_r)_{L^2} \quad \forall v \in V_r
%\end{equation}
We use the data set $\{u^n\}_{n=1}^N \subset V^h$ to compute the POD modes $\{\varphi_j\}_{j=1}^r \subset V^h$ using the new DQ approach with respect to the Hilbert space $X$. Below we take $X$ to be either $X = L^2(\Omega)$ or $X = H_0^1(\Omega)$. Let $V_r^h = \text{span}\{\varphi_j\}_{j=1}^r$. Next, we develop the POD reduced order model of the heat equation by substituting $u_r$ for the unknown $u$, using the Galerkin method, and projecting \Cref{BEheateq} onto the space $V^h_r \subset V^h$. Thus we arrive at the following BE-POD-ROM:
\begin{equation}\label{BE-POD-ROM}
\left( \frac{u_r^{n+1}-u_r^n}{\Delta t}, v_r \right)_{L^2} + \nu (\nabla u_r^{n+1} , \nabla v_r )_{L^2} = (f^{n+1} , v_r)_{L^2} \quad \forall v_r \in V^h_r.
\end{equation}
We split the error, in the standard way, with
\begin{equation*}
e^{n+1} = u^{n+1}-u_r^{n+1} = (u^{n+1}-\pi_r u^{n+1}) - (u_r^{n+1}-\pi_r u^{n+1}) = \eta^{n+1} - \phi_r^{n+1}
\end{equation*}
where $\pi_r$ is a projection onto $V_r^h$, $\eta^{n+1} = u^{n+1}-\pi_r u^{n+1}$ is the POD projection error, and $\phi_r^{n+1} = u_r^{n+1}-\pi_r u^{n+1}$ is the discretization error.
We subtract \Cref{BE-POD-ROM} from \Cref{BEheateq} and make the error substitution given above to get
\begin{equation}\label{eqn:preRitzprojheateqn}
\left( \frac{\phi_r^{n+1}-\phi_r^n}{\Delta t} , v_r \right)_{L^2} + \nu ( \nabla \phi_r^{n+1} , \nabla v_r )_{L^2} = \left( \frac{\eta^{n+1}-\eta^n}{\Delta t} , v_r \right)_{L^2} + \nu ( \nabla \eta^{n+1} , \nabla v_r )_{L^2}  \quad \forall v_r \in V_r^h.
\end{equation}

\begin{remark}
	The approach taken in this work is different than the one taken in \cite{Koc2020}. In that work the authors compare the ROM solution $u_r^n$ to the exact solution of the PDE $u(t_n)$. We can bound the error between the ROM solution and the exact solution using the triangle inequality
	\begin{equation*}
	\|u_r^n - u(t_n)\|_Y \leq \|u_r^n - u^n \|_Y + \|u^n - u(t_n)\|_Y
	\end{equation*}
	where $Y$ is a Hilbert space. The current work focuses on bounding the ROM error term $\|u_r^n - u^n \|_Y$ with $Y = L^2(\Omega)$ or $Y = H_0^1(\Omega)$ and leaves the second term to be studied using well-known finite element theory. 
\end{remark}

For analysis and computations the initial condition is taken to be the POD projection of the given initial condition, i.e, $u_r^1 = \Pi_r^X u^1$. Other initial conditions are possible and have been considered in other works for the standard POD and DQ POD approaches. For example, in \cite{Koc2020}, the Ritz projection was used for the initial condition. 
We now consider each POD space, $L^2(\Omega)$ and $H_0^1(\Omega)$, separately. 
\subsection{POD Space: $L^2(\Omega)$}
In this section, we take $X$ to be the space $L^2(\Omega)$. The orthogonal projection onto $V_r^h$, $\Pi_r^X:L^2(\Omega) \to L^2(\Omega)$, is given by
\begin{equation}
\Pi_r^X u = \sum_{i=1}^r (u, \varphi_i)_{L^2} \varphi_i
\end{equation}
and the set of POD modes $\{\varphi_i \}$ are orthogonal in $L^2(\Omega)$.  
%In this scenario we know the following error formulas.
%\begin{lemma}\label{lem:PODL2errformulas}
%	\begin{equation}
%	\sum_{j=1}^{s} \gamma_j \|u^n -\Pi_r^X u^n \|_{L^2}^2 = \sum_{k>r} \lambda_k
%	\end{equation}
%	and 
%	\begin{equation}
%	\sum_{j=1}^{s} \gamma_j \|u^n -\Pi_r^X u^n \|_{H_0^1}^2 = \sum_{k>r} \lambda_k \|\nabla\varphi_k \|_{L^2}^2
%	\end{equation}	
%\end{lemma}  
Define $\pi_r$ to be the Ritz projection $R_r$ which satisfies 
\begin{equation}
(\nabla (w-R_r w), \nabla v_r)_{L^2} = 0 
\end{equation}
for all $v_r \in V_r^h$. Thus, for all $w \in V^h$ we have
\begin{equation*}
\|w-R_r w \|_{H_0^1} = \inf_{v_r \in V_r^h} \|w-v_r\|_{H_0^1}.
\end{equation*}
Let $\eta_{Ritz}^{n+1} = u^{n+1}-R_r u^{n+1}$. Then \Cref{eqn:preRitzprojheateqn} becomes
\begin{equation}\label{eqn:L2Ritzprojheateqn}
\left( \frac{\phi_r^{n+1}-\phi_r^n}{\Delta t} , v_r \right)_{L^2} + \nu ( \nabla \phi_r^{n+1} , \nabla v_r )_{L^2} = \left( \frac{\eta_{Ritz}^{n+1}-\eta_{Ritz}^n}{\Delta t} , v_r \right)_{L^2} \quad \forall v_r \in V_r^h.
\end{equation}

Using this ROM error equation, we obtain the error bounds given below. Note, the constant $C$ can change from step to step, but it does not depend on any discretization parameter. In \Cref{ssec:Ratios_Comp_ROM} we investigate the final value of $C$ computationally.

\begin{theorem}\label{thm:L2_L2ROM}
	The pointwise $L^2$ solution error when the $L^2(\Omega)$ POD basis is used for the BE-POD-ROM is bounded by
	\begin{equation}\label{equ:L2_L2ROM}
	\max_k \|e^k\|^2_{L^2} \leq C \left( \sum_{i=r+1}^s \lambda_i^{DQ1} \|\varphi_i - R_r \varphi_i \|_{L^2}^2 + \|\phi_r^1 \|_{L^2}^2 \right).
	\end{equation}
\end{theorem}
\begin{proof}
	Taking $v_r = \phi_r^{n+1}$ in \Cref{eqn:L2Ritzprojheateqn} yields
	\begin{equation}
	\left( \frac{\phi_r^{n+1}-\phi_r^n}{\Delta t},\phi_r^{n+1} \right)_{L^2} + \nu\|\nabla \phi_r^{n+1} \|_{L^2}^2 =  \left( \frac{\eta_{Ritz}^{n+1}- \eta_{Ritz}^n}{\Delta t}, \phi_r^{n+1} \right)_{L^2}.
	\end{equation}
	Now, apply Cauchy-Schwartz followed by Young's inequality with a constant $\delta$ to get
%	\begin{equation}
%	\left( \frac{\phi_r^{n+1}-\phi_r^n}{\Delta t},\phi_r^{n+1} \right)_{L^2} + \nu\|\nabla \phi_r^{n+1} \|_{L^2}^2 \leq \|\partial \eta_{Ritz}^n\|_{L^2} \|\phi_r^{n+1}\|_{L^2}
%	\end{equation}
%	followed by Young's inequality to get
	\begin{equation}
	\left( \frac{\phi_r^{n+1}-\phi_r^n}{\Delta t},\phi_r^{n+1} \right)_{L^2} + \nu\|\nabla \phi_r^{n+1} \|_{L^2}^2 \leq \frac{1}{4\delta} \|\partial \eta_{Ritz}^n\|_{L^2}^2 + \delta \|\phi_r^{n+1}\|_{L^2}^2.
	\end{equation}
	Finally apply a polarization identity, use the fact that $C \|\phi_r^{n+1}\|_{L^2}^2 \leq \|\nabla \phi_r^{n+1}\|_{L^2}^2$, and rearrange to obtain
	\begin{equation}
	\left(\frac{1}{2\Delta t} - \delta +\frac{\nu}{C} \right) \|\phi_r^{n+1}\|_{L^2}^2 - \frac{1}{2\Delta t} \|\phi_r^n\|_{L^2}^2 \leq \frac{1}{4\delta} \|\partial \eta_{Ritz}^n\|_{L^2}^2.
	\end{equation}
	Taking $\delta = \frac{\nu}{C}$ and multiplying by $2\Delta t$ we get
	\begin{equation}
	 \|\phi_r^{n+1}\|_{L^2}^2 - \|\phi_r^n\|_{L^2}^2 \leq \frac{C\Delta t}{2\nu} \|\partial \eta_{Ritz}^n\|_{L^2}^2.
	\end{equation}
	Summing from $n=1$ to $n=k-1$, using $\|e^n \|_{L^2}^2 \leq 2 ( \|\eta^n\|_{L^2}^2 + \|\phi_r^n\|_{L^2}^2 )$, and rearranging once again gives
	\begin{equation}
	\|e^k\|_{L^2}^2 \leq C \left( \Delta t \sum_{n=1}^{k-1} \|\partial \eta_{Ritz}^n\|_{L^2}^2 + \|\eta_{Ritz}^k\|_{L^2}^2 + \|\phi_r^1\|_{L^2}^2 \right).
	\end{equation}
	Then apply \Cref{lem:POD_DQ1_error} and \Cref{thm:POD_DQ1_pointwise_error} with $Y = L^2(\Omega)$ and $\pi_r = R_r$ to get 
	\begin{equation}
	\|e^k\|_{L^2}^2 \leq C \left( \sum_{i=r+1}^s \lambda_i^{DQ1} \|\varphi_i - R_r \varphi_i \|_{L^2}^2 + \|\phi_r^1 \|_{L^2}^2 \right).
	\end{equation}
	Take the maximum over all $k$ to get the result. 
\end{proof}

\begin{theorem}\label{thm:H1_L2ROM}
	The pointwise $H_0^1$ solution error when the $L^2(\Omega)$ POD basis is used for the BE-POD-ROM is bounded by
	\begin{equation}\label{eqn:H1_L2ROM}
	\max_k \| \nabla e^k\|^2_{L^2} \leq C \left( \sum_{i=r+1}^{s} \lambda_i^{DQ1} \left( \|\varphi_i-R_r\varphi_i\|_{L^2}^2 + \|\nabla (\varphi_i-R_r\varphi_i)\|_{L^2}^2  \right) + \|\phi_r^1\|_{H_0^1}^2 \right).
	\end{equation}
\end{theorem}
\begin{proof}
	Taking $v_r = \partial \phi_r^{n}$ in \Cref{eqn:L2Ritzprojheateqn} 
	gives
	\begin{equation}
	\left( \frac{\phi_r^{n+1} - \phi_r^n}{\Delta t}, \partial \phi_r^{n} \right)_{L^2} + \nu \left( \nabla \phi_r^{n+1}, \nabla \partial \phi_r^{n} \right)_{L^2} = \left( \frac{\eta_{Ritz}^{n+1} - \eta_{Ritz}^n}{\Delta t}, \partial \phi_r^{n} \right)_{L^2}.
	\end{equation}
	Rearranging and using the definition of $\partial \phi_r^{n}$ yields
	\begin{equation}
	\nu \left( \nabla \phi_r^{n+1}, \nabla \partial \phi_r^{n} \right)_{L^2} = \left( \frac{\eta_{Ritz}^{n+1} - \eta_{Ritz}^n}{\Delta t}, \partial \phi_r^{n} \right)_{L^2} - \|\partial \phi_r^{n} \|_{L^2}^2.
	\end{equation}
	Then applying first the Cauchy-Schwartz inequality followed by Young's inequality with the constant $\delta$ we obtain 
	\begin{align*}
	\nu \left( \nabla \phi_r^{n+1}, \nabla \partial \phi_r^{n} \right)_{L^2} & \leq  \left( \|\partial \eta_{Ritz}^n \|_{L^2} \|\partial \phi_r^{n+1}\|_{L^2} \right) - \|\partial \phi_r^{n} \|_{L^2}^2 \\
	& \leq  \left( \frac{1}{4\delta}\|\partial \eta_{Ritz}^n \|_{L^2}^2 + \delta \|\partial \phi_r^{n}\|_{L^2}^2 \right) - \|\partial \phi_r^{n} \|_{L^2}^2  \\
	& = \frac{1}{4\delta} \|\partial \eta_{Ritz}^n \|_{L^2}^2 + (\delta-1) \|\partial \phi_r^{n} \|_{L^2}^2.
	\end{align*}
	Taking the constant $\delta = 1$ yields
	\begin{equation}\label{eqn:1}
	\nu \left( \nabla \phi_r^{n+1}, \nabla \partial \phi_r^{n} \right)_{L^2} \leq \frac{1}{4} \|\partial \eta_{Ritz}^n \|_{L^2}^2.
	\end{equation}	
%	Thus 
%	\begin{equation}
%	\Delta t  \left( \nabla \phi_r^{n+1} , \nabla \partial \phi_r^{n+1} \right)_{L^2} \leq \frac{C}{\nu} \Delta t \|\partial \eta_{Ritz}^n \|_{L^2}^2.
%	\end{equation}
	Also,
	\begin{equation}\label{eqn:2}
	\Delta t \left( \nabla \phi_r^{n+1} , \nabla \partial \phi_r^{n} \right)_{L^2} = \frac{1}{2} \left( \|\nabla \phi_r^{n+1} \|_{L^2}^2 + \|\nabla \phi_r^{n+1}- \nabla\phi_r^{n}\|_{L^2}^2  - \|\nabla \phi_r^{n}\|_{L^2}^2 \right).
	\end{equation}
	Combining \Cref{eqn:1} and \Cref{eqn:2}, summing over $n = 1,\ldots ,k-1$, and using $\|\nabla e^k\|_{L^2}^2 \leq 2\left( \|\nabla \eta^k\|_{L^2}^2 + \|\nabla \phi_r^k\|_{L^2}^2  \right) $ gives
	\begin{align}
		\|\nabla e^k \|_{L^2}^2 %&\leq C \left( \Delta t \sum_{n=1}^{k-1}\|\partial \eta_{Ritz}^n\|_{L^2}^2 + \|\nabla \eta_{Ritz}^k \|_{L^2}^2 + \|\nabla \phi_r^1\|_{L^2}^2 \right) \\
		& \leq C \left( \Delta t \sum_{n=1}^{k-1}\|\partial \eta_{Ritz}^n\|_{L^2}^2 + \| \eta_{Ritz}^k \|_{H_0^1}^2 + \| \phi_r^1\|_{H_0^1}^2 \right).
	\end{align}
		Then apply \Cref{lem:POD_DQ1_error} and \Cref{thm:POD_DQ1_pointwise_error} with $Y = H_0^1(\Omega)$ and $\pi_r = R_r$ to get 
		\begin{equation}
		\|\nabla e^k \|_{L^2}^2 \leq C \left( \sum_{i=r+1}^{s} \lambda_i^{DQ1} \|\varphi_i-R_r\varphi_i\|_{L^2}^2 + \sum_{i=r+1}^{s} \lambda_i^{DQ1} \|\varphi_i-R_r\varphi_i\|_{H_0^1}^2  + \|\phi_r^1\|_{H_0^1}^2 \right).
		\end{equation}
		Rearrange and take the maximum over all $k$ to get the result. 
\end{proof}
%Note: See also \cite[Lemma 4.5]{Koc2020}.

\begin{theorem}\label{thm:sol_L2ROM}
	The pointwise solution norm error when the $L^2(\Omega)$ POD basis is used for the BE-POD-ROM is bounded by
	\begin{equation}
	 \|e^{N}\|_{L^2}^2 +\nu \Delta t \sum_{n=1}^{N-1} \|\nabla e^{n+1}\|_{L^2}^2  \leq C \left( \sum_{i=r+1}^{s} \lambda^{DQ1} (\|\varphi_i-R_r\varphi_i\|_{L^2}^2 + \|\varphi_i-R_r\varphi_i\|_{H_0^1}^2) + \|\phi_r^1\|_{L^2}^2\right).
	\end{equation}
\end{theorem}

\begin{proof}
	Taking $\phi_r^{n+1}$ in \Cref{eqn:L2Ritzprojheateqn}, yields
	\begin{equation}
	\left( \frac{\phi_r^{n+1}-\phi_r^n}{\Delta t},\phi_r^{n+1} \right)_{L^2} + \nu\|\nabla \phi_r^{n+1} \|_{L^2}^2 =  \left( \frac{\eta_{Ritz}^{n+1}- \eta_{Ritz}^n}{\Delta t}, \phi_r^{n+1} \right)_{L^2}.
	\end{equation}
	Now, apply the Cauchy-Schwartz inequality, Young's inequality, and a polarization identity to get
	\begin{equation}
	\frac{1}{2\Delta t} \left( \|\phi_r^{n+1}\|_{L^2}^2 - \|\phi_r^n\|_{L^2}^2 \right) + \nu \|\nabla \phi_r^{n+1}\|_{L^2}^2 \leq \frac{1}{4\delta} \|\partial \eta_{Ritz}^n\|_{L^2}^2 + \delta \|\phi_r^{n+1}\|_{L^2}^2.
	\end{equation}
	Now multiply by 2, apply the Poincar\'e inequality to the last term, and combine the resulting like terms:
	\begin{equation}
	\frac{1}{\Delta t} \left( \|\phi_r^{n+1}\|_{L^2}^2 - \|\phi_r^n\|_{L^2}^2 \right) + 2\left(\nu - \delta C \right) \|\nabla \phi_r^{n+1}\|_{L^2}^2 \leq \frac{1}{2\delta} \|\partial \eta_{Ritz}^n\|_{L^2}^2.
	\end{equation}
	Take $\delta = \frac{\nu}{2C}$ and multiply by $\Delta t$:
	\begin{equation}
	 \|\phi_r^{n+1}\|_{L^2}^2 - \|\phi_r^n\|_{L^2}^2  + \nu \Delta t \|\nabla \phi_r^{n+1}\|_{L^2}^2 \leq \frac{C \Delta t}{\nu} \|\partial \eta_{Ritz}^n\|_{L^2}^2.
	\end{equation}
	Sum from $n =1$ to $k-1$, rearrange, and take a maximum amoung the constants to obtain
	\begin{equation}
	\|\phi_r^{k}\|_{L^2}^2 +\nu \Delta t \sum_{n=1}^{k-1} \|\nabla \phi_r^{n+1}\|_{L^2}^2 \leq C \left( \Delta t  \sum_{n=1}^{k-1} \|\partial \eta_{Ritz}^n\|_{L^2}^2+ \|\phi_r^1\|_{L^2}^2 \right).
	\end{equation}
	Using $\|e^n \|_{L^2}^2 \leq 2 ( \|\eta^n\|_{L^2}^2 + \|\phi_r^n\|_{L^2}^2 )$ and rearranging gives
	\begin{equation}
	\|e^{k}\|_{L^2}^2 +\nu \Delta t \sum_{n=1}^{k-1} \| e^{n+1}\|_{H_0^1}^2 \leq C \left( \Delta t  \sum_{n=1}^{k-1} \|\partial \eta_{Ritz}^n\|_{L^2}^2+\|\eta_{Ritz}^{k}\|_{L^2}^2+ \Delta t \sum_{n=1}^{k-1} \| \eta_{Ritz}^{n+1}\|_{H_0^1}^2 + \|\phi_r^1\|_{L^2}^2  \right).
	\end{equation}
	Finally, using \Cref{lem:POD_DQ1_error},  \Cref{thm:POD_DQ1_pointwise_error}, and \Cref{cor:totalbounds} with $\pi_r = R_r$ and both $Y = L^2(\Omega) $ and $Y= H_0^1(\Omega)$ yield
	\begin{equation}
	\|e^{k}\|_{L^2}^2 +\nu \Delta t \sum_{n=1}^{k-1} \| e^{n+1}\|_{H_0^1}^2 \leq C \left( \sum_{i=r+1}^{s} \lambda^{DQ1} (\|\varphi_i-R_r\varphi_i\|_{L^2}^2 + \|\varphi_i-R_r\varphi_i\|_{H_0^1}^2) + \|\phi_r^1\|_{L^2}^2\right).
	\end{equation}
	Taking $k=N$ yields the result.
\end{proof}

\subsection{POD Space: $H^1_0(\Omega)$}
Alternatively, we can take the POD space $X$ to be $H^1_0(\Omega)$. The orthogonal projection onto $V_r^h$, $\Pi_r^X:H^1_0(\Omega) \to H^1_0(\Omega)$, is given by
\begin{equation}
\Pi_r^X u = \sum_{i=1}^r (u, \varphi_i)_{H^1_0} \varphi_i
\end{equation}
and the set of POD modes $\{\varphi_i \}$ are orthogonal in $H^1_0(\Omega)$. 
%We can also consider the projection $\Pi_r^X: Y \to Y$. This projection \blue{might be?} is non-orthogonal but still projects onto $X_r$. 
%In this scenario we know the following error formulas.
%\begin{lemma}\label{lem:PODH1errformulas}
%	\begin{equation}
%	\sum_{j=1}^{s} \gamma_j \|u^n -\Pi_r^X u^n \|_{H^1_0}^2 = \sum_{k>r} \lambda_k
%	\end{equation}
%	and 
%	\begin{equation}
%	\sum_{j=1}^{s} \gamma_j \|u^n -\Pi_r^X u^n \|_{L^2}^2 = \sum_{k>r} \lambda_k \|\varphi_k \|_{L^2}^2
%	\end{equation}	
%\end{lemma}  
Note that $\Pi_r^X \varphi_k = \sum_{i=1}^r (\varphi_k, \varphi_i)_{H^1_0} \varphi_i = 0$ for $k >r$ since $\{ \varphi_k \}$ are orthogonal in $H^1_0(\Omega)$. Further, since $\Pi_r^X$ is orthogonal we have
\begin{align*}
& (w-\Pi_r^X w , v_r)_{X} = 0 \quad \forall v_r \in V_r^h \\
\implies & (w-\Pi_r^X w , v_r)_{H^1_0} = 0 \quad \forall v_r \in V_r^h \\
\implies & ( \nabla (w-\Pi_r^X w), \nabla v_r )_{L^2} = 0 \quad \forall v_r \in V_r^h.
\end{align*}
Similarly to before we take $\pi_r$ to be the Ritz projection, but in this case the Ritz projection is the orthogonal projection $\Pi_r^X$, and \Cref{eqn:preRitzprojheateqn} becomes
\begin{equation}
\left( \frac{\phi_r^{n+1}-\phi_r^n}{\Delta t} , v_r \right)_{L^2} + \nu ( \nabla \phi_r^{n+1} , \nabla v_r )_{L^2} = \left( \frac{\eta_{Ritz}^{n+1}-\eta_{Ritz}^n}{\Delta t} , v_r \right)_{L^2}  \quad \forall v_r \in V_r^h
\end{equation}
where $\eta_{Ritz}^{n+1} = u^{n+1} - R_r u^{n+1} = u^{n+1} - \Pi_r^X u^{n+1}$.
As before we can show error bounds for the $L^2$ and $H_0^1$ norm errors and the solution norm error.

\begin{theorem}
	The ROM solution errors when the $H_0^1(\Omega)$ POD basis is used for the BE-POD-ROM are bounded as follows:
	\begin{equation}\label{eqn:L2_H1ROM}
	\max_{k} \|e^k\|^2_{L^2} \leq C \left( \sum_{i=r+1}^s \lambda_i^{DQ1} \|\varphi_i \|_{L^2}^2 + \|\phi_r^1 \|_{L^2}^2 \right),
	\end{equation}
	\begin{equation}\label{eqn:H1_H1ROM}
	\max_{k} \|\nabla e^k\|^2_{L^2} \leq C \left( \sum_{i=r+1}^{s} \lambda_i^{DQ1} (\|\varphi_i\|_{L^2}^2 + 1) +\|\phi_r^1\|_{H_0^1}^2 \right),
	\end{equation}
	and
	\begin{equation}
	 \|e^{N}\|_{L^2}^2 +\nu \Delta t \sum_{n=1}^{N-1} \|\nabla e^{n+1}\|_{L^2}^2  \leq C \left( \sum_{i=r+1}^{s} \lambda^{DQ1} (1 + \|\varphi_i\|_{L^2}^2) + \|\phi_r^1\|_{L^2}^2\right).
	\end{equation}
\end{theorem}

\begin{proof}
	Proceed as in the proofs of \Cref{thm:L2_L2ROM}, \Cref{thm:H1_L2ROM}, and \Cref{thm:sol_L2ROM} respectively, but now use $\eta_{Ritz}= u - \Pi_r^X u$ and the appropriate choices for $Y$ in \Cref{thm:POD_DQ1_pointwise_error}, \Cref{cor:totalbounds}, and \Cref{lem:POD_DQ1_error}.
\end{proof}

\subsection{Optimality}\label{ssec:optimality}
In this section we investigate the optimality of this new approach to POD with difference quotients. To do so we follow the approach given in \cite{Koc2020} but modify the optimality definitions given there to include the ROM error for the initial condition. For this work we focus on both of the POD ROM discretization errors and assume the time and spatial discretization errors are optimal. Thus, we ignore the latter errors in the equation given below. In this section we consider only the new approach as the standard difference quotient approach is discussed in \cite{Koc2020}. The optimality of each error depends on both the POD space $X$ and the error norm space $Y$. To give a precise definition of pointwise ROM error optimality, assume there exists a constant $ C > 0 $ that is independent of the discretization parameters such that the ROM errors $ e^k = u^k - u^k_r $ for $ k = 1, \ldots, N $ satisfy
\begin{equation}
\max_{1 \leq k \leq N} \|e^k\|_Y^2 \leq C\left( \Lambda_r +\Lambda_r^1 \right)
%+ \zeta(\Delta t) + \xi(h) \right)
\end{equation}
where
\begin{itemize}
	\item $\Lambda_r$ is the ROM discretization error, and depends only on $r$, the POD eigenvalues, and the POD modes;
	\item $\Lambda_r^1$ is the ROM discretization error for the initial condition, and depends only on $r$, the POD eigenvalues, and the POD modes.
%	\item $\zeta(\Delta t)$ is an optimal time discretization error; and 
%	\item $\xi(h)$ is an optimal spatial discretization error. 
\end{itemize}

Let $X_r \subset X$ be the span of the first $r$ POD modes, and assume $X_r \subset Y$. Let $\Pi_r^X:X \to X$ be the orthogonal projection onto $X_r$ and $\Pi_r^Y: Y \to Y$ be the $Y-$orthogonal projection onto $X_r$. Further let $s$ be the number of positive POD eigenvalues. In \Cref{def:optimality} we extend the definitions of optimality provided in \cite{Koc2020} to include the ROM discretization error for the initial condition. 
\begin{definition}\label{def:optimality}
	We say the total ROM discretization error, $\Lambda_r+ \Lambda_r^1$, is
	\begin{itemize}
%		\item \textbf{truly optimal} if there exists a constant $C$ such that
%		\begin{equation}
%		\Lambda_r+ \Lambda_r^1 \leq C \Lambda_r^*, \; \Lambda_r^* := \max_{k} \|u^k - \Pi_r^Y u^k \|_Y^2,
%		\end{equation}
		\item \textbf{optimal-I} if there exists a constant $C$ such that 
		\begin{equation}
		\Lambda_r+ \Lambda_r^1 \leq C \sum_{i=r+1}^s \lambda_i\|\varphi_i\|_Y^2, %\; \Lambda_r^I := \sum_{i=r+1}^s \lambda_i\|\varphi_i\|_Y^2,
		\end{equation}
		\item \textbf{optimal-II} if there exists a constant $C$ such that 
		\begin{equation}
		\Lambda_r+ \Lambda_r^1 \leq C \sum_{i=r+1}^s \lambda_i\|\varphi_i- \Pi_r^Y\varphi_i\|_Y^2. %\; \Lambda_r^{II} := \sum_{i=r+1}^s \lambda_i\|\varphi_i- \Pi_r^Y\varphi_i\|_Y^2.
		\end{equation}
	\end{itemize}
	The constant $C$ can depend on the solution data and the problem data, but does not depend on any discretization parameter. 
\end{definition}
\begin{remark}
	For a detailed discussion on the optimality types see \cite{Koc2020}. Note that, as shown in \cite[Proposition 4.8]{Koc2020}, optimal-II is stronger than optimal-I, and the two are equivalent if $X = Y$. 
	%Note that the ROM discretization error for the initial condition is not considered there, but is included in this work. 
\end{remark}

First we consider the ROM error from the choice of initial condition by noting $\Lambda_r^1 = \|\phi_r^1\|_Y^2$ in \Cref{lem:initial condition} below.

\begin{lemma}\label{lem:initial condition}
	Let the initial condition be the POD projection of the given initial condition, i.e. $u_r^1 = \Pi_r^X u^1$ so that $\phi_r^1 =u_r^1 - R_r u^1= \Pi_r^X u^1 - R_r u^1$. If $X = L^2(\Omega)$ then	
	\begin{equation}\label{eqn:int_L2_1}
	\|\phi_r^1\|_{L^2}^2 \leq  2\sum_{i=r+1}^{s} \lambda_i + 2\sum_{i=r+1}^{s} \lambda_i \|\varphi_i - R_r \varphi_i \|_{L^2}^2 
	\end{equation}
	and
	\begin{equation}\label{eqn:int_L2_2}
		 \|\phi_r^1\|_{H_0^1}^2 \leq  2\sum_{i=r+1}^{s} \lambda_i \| \varphi_i\|_{H_0^1}^2 +  2\sum_{i=r+1}^{s} \lambda_i \|\varphi_i - R_r \varphi_i \|_{H_0^1}^2.
	\end{equation}
	If $X = H_0^1(\Omega)$, then
	\begin{equation}\label{eqn:int_H10}
	\phi_r^1 = 0
	\end{equation}
	for either choice of space $Y$.
\end{lemma}

\begin{proof}
	First consider the case $X = L^2(\Omega)$. We have
	\begin{align*}
	\|\phi_r^1\|_{Y}^2 &= \|u_r^1 - R_r u^1 \|_{Y}^2 \\
	& = \|\Pi_r^X u^1 - R_r u^1\|_{Y}^2 \\
	& \leq 2\|\Pi_r^X u^1 - u^1\|_{Y}^2 + 2\|u^1 - R_r u^1\|_{L^2}^2.
	\end{align*}
	Using \Cref{lem:POD_DQ1_error} with $Y = X = L^2(\Omega)$ and $\pi_r = R_r$ gives \Cref{eqn:int_L2_1} and with $Y = H_0^1(\Omega)$ and $\pi_r = R_r$ gives \Cref{eqn:int_L2_2}. Finally if $X = H_0^1(\Omega)$, then $\Pi_r^X = R_r$ and \Cref{eqn:int_H10} follows by definition of $\pi_r$.
\end{proof}

This lemma allows us to investigate the total ROM discretization error in the following two theorems with the POD space taken to be $L^2$ and $H_0^1$ respectively.

\begin{theorem}\label{thm:L2_opt} If the $L^2$ POD basis is used, i.e. $X = L^2(\Omega)$, then the following hold:
	\begin{itemize}
		\item The pointwise ROM error in \Cref{equ:L2_L2ROM} with the error norm $Y = L^2$ is optimal-I if there exists a constant $C$ such that 
		\begin{equation}\label{equ:constantbound_L2}
		\|\varphi_i - R_r \varphi_i \|_{L^2} \leq C
		\end{equation}
		for $r+1 \leq i \leq s$. Note that in this case optimal-I is equivalent to optimal-II.
		\item The pointwise ROM error in \Cref{eqn:H1_L2ROM} with error norm $Y = H_0^1$ is optimal-I.
	\end{itemize} 
\end{theorem}

\begin{remark}\label{remark_opt}
	In \cite[Theorem 4.10 (iv)]{Koc2020}, when using the $L^2$ POD basis with $Y = H_0^1(\Omega)$ but not considering the initial condition error the ROM error was optimal-II. When including the initial condition as the POD projection of the given initial condition, we obtain the weaker result that the error is optimal-I. Other choices of initial condition will yield different results. Further, the assumption \Cref{equ:constantbound_L2} is discussed in greater detail in \cite[Section 4.2]{Koc2020}. 
\end{remark} 

\begin{proof}
	For the first item
	\begin{align*}
	\Lambda_r +\Lambda_r^1 &= \sum_{i=r+1}^s \lambda_i \|\varphi_i - R_r \varphi_i \|_{L^2}^2 + \|\phi_r^1 \|_{L^2}^2 \\
	& \leq \sum_{i=r+1}^s \lambda_i \|\varphi_i - R_r \varphi_i \|_{L^2}^2 +  2 \left(\sum_{i=r+1}^{s} \lambda_i + \sum_{i=r+1}^{s} \lambda_i \|\varphi_i - R_r \varphi_i \|_{L^2}^2 \right)  \\
	%& \leq  C \left( \sum_{i=r+1}^{s} \lambda_i + \sum_{i=r+1}^{s} \lambda_i \|\varphi_i - R_r \varphi_i \|_{L^2}^2 \right) \\
	%& \leq \sum_{i=r+1}^s \lambda_i \|\varphi_i - R_r \varphi_i \|_{L^2}^2  + C \sum_{i=r+1}^s \lambda_i \\ %\|\varphi_i - R_r \varphi_i \|_{L^2}^2 \\
	& \leq C \sum_{i=r+1}^s \lambda_i\\
	& = C \sum_{i=r+1}^s \lambda_i \|\varphi_i\|_{L^2}^2
	%& \leq C\sum_{i=r+1}^s \lambda_i \|\varphi_i\|_{H_0^1}^2 \\
	%& = C \Lambda_r^I
	\end{align*}
	using \Cref{lem:initial condition}, the assumption in \Cref{equ:constantbound_L2},  \Cref{thm:POD_DQ1_pointwise_error}, and the $L^2$ orthonormality of the POD basis.
%	Note: Since we chose to use the standard POD projection for our initial condition
%	\begin{align*}
%	\|u_r^1 - R_r u^1 \|_{L^2}^2 & = \|\Pi_r^X u^1 - R_r u^1\|_{L^2}^2 \\
%	& \leq \|\Pi_r^X u^1 - u^1\|_{L^2}^2 + \|u^1 - R_r u^1\|_{L^2}^2 \\
%	& \leq C \sum \lambda_i + C \sum \lambda_i \|\varphi_i - R_r \varphi_i \|_{L^2}^2 \\
%	& \leq C \sum \lambda_i
%	\end{align*}

	For the second item use the Poincar\'e inequality to get, 
	\begin{align*}
	\Lambda_r +\Lambda_r^1 %&= \sum_{i=r+1}^{s} \lambda_i (\|\varphi_i-R\varphi_i\|_{L^2}^2 + \|\nabla(\varphi_i-R\varphi_i)\|_{L^2}^2) + \|\phi_r^1\|_{H_0^1}^2\\
	& = \sum_{i=r+1}^{s} \lambda_i \|\varphi_i-R_r\varphi_i\|_{L^2}^2 + \sum_{i=r+1}^{s} \lambda_i \|\nabla(\varphi_i-R_r\varphi_i)\|_{L^2}^2 +  \|\phi_r^1\|_{H_0^1}^2 \\
%	& \leq  C\sum_{i=r+1}^{s} \lambda_i \|\nabla(\varphi_i-R\varphi_i)\|_{L^2}^2 + \| \phi_r^1\|_{H_0^1}^2 \\
	& \leq C\sum_{i=r+1}^{s} \lambda_i \|\varphi_i-R_r\varphi_i\|_{H_0^1}^2 + \| \phi_r^1\|_{H_0^1}^2.
	\end{align*}
%	Now we consider the $\| \phi_r^1\|_{H_0^1}^2$ term. 
%	\begin{align*}
%	\| \phi_r^1\|_{H_0^1}^2 &= \|u_r^1 - R_r u^1\|_{H_0^1}^2 \\
%	& \leq \|u_r^1 - u^1 \|_{H_0^1}^2 + \|u^1 - R_r u^1 \|_{H_0^1}^2 \\
%	& =  \|\Pi_r^X u^1 - u^1 \|_{H_0^1}^2 + \|u^1 - R_r u^1 \|_{H_0^1}^2 \\
%	& \leq C \sum_{i=r+1}^{s} \lambda_i \| \varphi_i\|_{H_0^1}^2 + C \sum \lambda_i \|\varphi_i - R_r \varphi_i \|_{H_0^1}^2
%	\end{align*}
	\Cref{lem:initial condition} and the orthogonality of the Ritz projection $R_r : H^1_0 \to H^1_0$ then yields
%	\begin{align*}
%	\Lambda_r +\Lambda_r^1 & \leq C \sum_{i=r+1}^s \lambda_i \|\varphi_i - R_r \varphi_i \|_{H_0^1}^2 + C \sum_{i=r+1}^{s} \lambda_i \| \varphi_i\|_{H_0^1}^2 \\
%	& \leq C \sum_{i=r+1}^s \lambda_i \| \varphi_i\|_{H_0^1}^2
%	\end{align*}
	\begin{equation*}
	\Lambda_r +\Lambda_r^1 \leq C \sum_{i=r+1}^s \lambda_i \|\varphi_i - R_r \varphi_i \|_{H_0^1}^2 + C \sum_{i=r+1}^{s} \lambda_i \| \varphi_i\|_{H_0^1}^2  \leq C \sum_{i=r+1}^s \lambda_i \| \varphi_i\|_{H_0^1}^2.
	\end{equation*}
	Thus the error given in \Cref{eqn:H1_L2ROM} is optimal-I. 
\end{proof}

\begin{theorem}\label{thm:H10_opt} If the $H_0^1$ POD basis is used, i.e. $X = H_0^1$, then the following results hold. 
	\begin{itemize}
		\item The pointwise ROM error in \Cref{eqn:L2_H1ROM} with the error norm $Y = L^2$ is optimal-I.  
		\item The pointwise ROM error in \Cref{eqn:L2_H1ROM} with error norm $Y = L^2$ is also optimal-II when 
		\begin{equation*}
		\| \varphi_i\|_{Y} \leq C \|\varphi_i - \Pi_r^Y \varphi_i \|_Y \quad  \mbox{for $ r+1 \leq i \leq s $.}
		\end{equation*} 
		\item The pointwise ROM error in \Cref{eqn:H1_H1ROM} with the error norm $Y = H_0^1$ is optimal-I. Note that in this case optimal-I is equivalent to optimal-II.
	\end{itemize}
\end{theorem}

The proof of this result is similar to that of \cite[Theorem 4.10]{Koc2020} since from \Cref{lem:initial condition} we know $\phi_r^1 = 0$ for both $Y = L^2$ and $Y = H_0^1$.  We omit the details. We note that the optimality results for the new POD DQ approach in \Cref{thm:L2_opt} and \Cref{thm:H10_opt} are very similar to \cite[Theorem 4.10]{Koc2020} for the standard POD DQ approach. In fact, the only fundamental difference in the results here was caused by our choice of the initial condition and including this in the optimality definition, as discussed in \Cref{remark_opt}.

%\begin{proof}
%	From \Cref{lem:initial condition} we know $\phi_r^1 = 0$ for both $Y = L^2$ and $Y = H_0^1$. Then proofs are the same as in the previous paper that did not consider the error from the initial condition. 
%\end{proof}

%These results differ slightly from those presented in \cite{Koc2020} because of the inclusion of the initial condition. Other choices of initial condition will yield different results. 

\section{Numerical Results}\label{sec:rom_computations}

We now turn our attention to computational results. In this section we use the same test problem as in \Cref{sec:some_comp_1} to compute ROM solution errors in order to compare the new POD approach to the existing standard approaches.  We also compute the scaling factors present in the bounds of the theoretical results from \Cref{thm:POD_DQ1_pointwise_error} and \Cref{sec:ROM}. 

\subsection{ROM Comparisons for the Three POD Approaches}\label{sec:rom_computations_comp} 
First we find the ROM errors for the new POD DQ approach at the final time, i.e., the errors at $T = 1$. We compute both the $L^2$ norm error and the $H^1_0(\Omega)$ norm error. For this computation we use 100 finite element nodes, 100 time steps and 3 different values for $r$. Note that in all tables below we report the square of the norms for consistency with previous results. For comparison we also show the final time errors for the standard POD approach and the standard DQ POD approach. First we consider the case where the POD space is taken to be $X = L^2(\Omega)$. The results for the $L^2$ and $H_0^1$ errors can be found in \Cref{tab:L2_errors} and \Cref{tab:H10_errors} respectively. In \Cref{tab:solnorm_errors_H10}, we compute the solution norm squared errors for this reduced order model using the same computation parameters. This solution norm error at the final time is given by \Cref{eqn:solution_norm_error}:
\begin{equation}\label{eqn:solution_norm_error}
 \|e^{N}\|_{L^2}^2 +\Delta t \sum_{n=1}^{N-1} \|\nabla e^{n+1}\|_{L^2}^2.
\end{equation}

\begin{table}
	\renewcommand{\arraystretch}{1.25}
	\begin{center}	
		\begin{tabular}{|c|c|c|c|c|}
			\hline
			POD Space & $r$ value & Standard POD & Standard DQ POD & New DQ POD \\
			\hline
			$L^2$ & 4 & 1.639e-08 & 1.206e-07 & 1.141e-07 \\
			& 6 & 1.257e-10 & 1.148e-09 & 1.090e-09 \\
			& 8 & 9.486e-13 & 9.144e-12 & 8.722e-12 \\
			\hline
			$H_0^1$ & 4 & 1.851e-08 & 1.350e-07 & 1.276e-07 \\
			& 6 & 1.324e-10 & 1.218e-09 & 1.156e-09 \\
			& 8 & 9.999e-13 & 9.778e-12 & 9.323e-12 \\
			\hline
		\end{tabular}
		\caption{ROM errors for the $L^2$ Norm at the final time \label{tab:L2_errors}}
	\end{center}
\end{table}

\begin{table}
	\renewcommand{\arraystretch}{1.25}
	\begin{center}	
		\begin{tabular}{|c|c|c|c|c|}
			\hline
			POD Space & $r$ value & Standard POD & Standard DQ POD & New DQ POD \\
			\hline
			$L^2$ & 4 & 1.936e-07 & 1.340e-06 & 1.269e-06 \\
			& 6 & 2.615e-09 & 2.364e-08 & 2.245e-08 \\
			& 8 & 2.264e-11 & 2.123e-10 & 2.024e-10 \\
			\hline
			$H_0^1$ & 4 & 2.171e-07 & 1.494e-06 & 1.413e-06 \\
			& 6 & 2.755e-09 & 2.506e-08 & 2.379e-08 \\
			& 8 & 2.396e-11 & 2.263e-10 & 2.157e-10 \\
			\hline
		\end{tabular}
		\caption{ROM errors for the $H_0^1$ Norm at the final time \label{tab:H10_errors}}
	\end{center}
\end{table}

\begin{table}
	\renewcommand{\arraystretch}{1.25}
	\begin{center}	
		\begin{tabular}{|c|c|c|c|c|}
			\hline
			POD Space & $r$ value & Standard POD & Standard DQ POD & New DQ POD \\
			\hline
			$L^2$ & 4 & 1.172e-07 & 2.441e-06 & 2.008e-06 \\
			& 6 & 1.472e-11 & 4.373e-10 & 3.576e-10 \\
			& 8 & 7.032e-16 & 2.227e-14 & 1.810e-14 \\
			\hline
			$H_0^1$ & 4 & 1.135e-07 & 3.049e-06 & 2.525e-06 \\
			& 6 & 1.467e-11 & 4.861e-10 & 3.996e-10 \\
			& 8 & 6.984e-16 & 2.479e-14 & 2.032e-14 \\
			\hline
		\end{tabular}
		\caption{ROM solution norm error \Cref{eqn:solution_norm_error} \label{tab:solnorm_errors_H10}}
	\end{center}
\end{table}

We also wish to compare the different POD approaches when we take the POD space to be $H_0^1(\Omega)$, i.e., $X = H_0^1(\Omega)$. All other parameters of the computation stay the same and the results can also be found in \Cref{tab:L2_errors}, \Cref{tab:H10_errors}, and \Cref{tab:solnorm_errors_H10}. For both cases of the chosen POD space the errors behave similarly in all three approaches but are particularly close in the two approaches that utilize the difference quotients. 

%\begin{table}
%	\begin{center}	
%		\begin{tabular}{|c|c|c|c|}
%			\hline
%			$r$ value & Standard POD & Standard DQ POD & New DQ POD \\
%			\hline
%			4 & 1.851e-08 & 1.350e-07 & 1.276e-07 \\
%			6 & 1.324e-10 & 1.218e-09 & 1.156e-09 \\
%			8 & 9.999e-13 & 9.778e-12 & 9.323e-12 \\
%			\hline
%		\end{tabular}
%		\caption{ROM errors for the $L^2$ Norm at the final time with POD space $H_0^1$ \label{tab:L2_errors_H10}}
%	\end{center}
%\end{table}

%\begin{table}
%	\begin{center}	
%		\begin{tabular}{|c|c|c|c|}
%			\hline
%			$r$ value & Standard POD & Standard DQ POD & New DQ POD \\
%			\hline
%			4 & 2.171e-07 & 1.494e-06 & 1.413e-06 \\
%			6 & 2.755e-09 & 2.506e-08 & 2.379e-08 \\
%			8 & 2.396e-11 & 2.263e-10 & 2.157e-10 \\
%			\hline
%		\end{tabular}
%		\caption{ROM errors for the $H_0^1$ Norm at the final time with POD space $H_0^1$\label{tab:H10_errors_H10}}
%	\end{center}
%\end{table}

%\begin{table}
%	\begin{center}	
%		\begin{tabular}{|c|c|c|c|}
%			\hline
%			$r$ value & Standard POD & Standard DQ POD & New DQ POD \\
%			\hline
%			4 & 1.135e-07 & 3.049e-06 & 2.525e-06 \\
%			6 & 1.467e-11 & 4.861e-10 & 3.996e-10 \\
%			8 & 6.984e-16 & 2.479e-14 & 2.032e-14 \\
%			\hline
%		\end{tabular}
%		\caption{Solution Norm Error: $L^2(0,T,H_0^1(\Omega))$ with POD space $H_0^1$ \label{tab:solnorm_errors_H10_H10}}
%	\end{center}
%\end{table}

\subsection{Pointwise Error Bounds}\label{ssec:Ratios_Comp_PW}

We want to compute the ratios generated by the theoretical results in \Cref{sec:PointwiseErrorBounds}. For these results we vary the number of time steps, while keeping everything else constant. This allows us to verify that the scaling factors are not dependent on the time step chosen. We performed similar experiments by varying other parameters and obtained similar results. We use 100 finite element nodes and an $r$ value of 4. Both $X = L^2(\Omega)$ and $X = H_0^1(\Omega)$ are considered in this section. We use the data generated by the finite element method as described in \Cref{sec:some_comp_1} and compute the POD modes and POD singular values.

First, we consider the pointwise error bounds as in \Cref{thm:POD_DQ1_pointwise_error}. The tables show the projection and the norm space $Y$ that is used for each of the computations. Note that for these we have $\pi_r = R_r$ and either $Y = L^2(\Omega)$ or $Y = H_0^1(\Omega)$. For example, for the second result in \Cref{thm:POD_DQ1_pointwise_error} with $X = L^2(\Omega)$ and $Y = H_0^1(\Omega)$ we have
\begin{equation*}
\mathrm{scaling \, factor} = \left( \max_j \|u^j - \Pi_r^X u^j\|_Y^2 \right)/ \left( \sum_{i=r+1}^{s}\lambda_i^{DQ1} \|\varphi_i\|_Y^2 \right).
\end{equation*}
%
%\begin{equation}
%C_{XY} = 
%\end{equation}
%
The results for $X = L^2(\Omega)$ and $X = H_0^1(\Omega)$ are given in \Cref{tab:ratios_pointwise_errors_L2} and \Cref{tab:ratios_pointwise_errors_H10} respectively. We present the results with fixed values of $r=4$ and $100$ finite element nodes and varying values of $\Delta t$. Theoretically, we showed that the scaling factor should always be less that or equal to $2$ (since $T=1$ here). The computational results show that the scaling factor can be much less than that value. 

\begin{table}
	\renewcommand{\arraystretch}{1.25}
	\begin{center}	
		\begin{tabular}{|c|c|c|c|c|c|c|}
			\hline
			Projection & $Y$ & $1/40$ & $1/50$ & $1/100$ & $1/200$ & $1/300$ \\
			\hline
			$\Pi_r^X$ & $L^2(\Omega)$ & 6.0e-02 & 5.4e-02 & 4.0e-02 & 3.2e-02 & 2.9e-02   \\
			$\Pi_r^X$ & $H_0^1(\Omega)$ & 5.6e-02 & 4.9e-02 & 3.3e-02 & 2.3e-02 & 1.8e-02   \\
			$R_r$ & $L^2(\Omega)$ & 5.8e-02 & 5.2e-02 & 3.7e-02 & 2.9e-02 & 2.5e-02   \\
			$R_r$ & $H_0^1(\Omega)$ & 5.5e-02 & 4.9e-02 & 3.2e-02 & 2.1e-02 & 1.6e-02   \\
			\hline
		\end{tabular}
		\caption{Scaling factors for \Cref{thm:POD_DQ1_pointwise_error} as $\Delta t $ changes with $X = L^2(\Omega)$ \label{tab:ratios_pointwise_errors_L2}}
	\end{center}
\end{table}

\begin{table}
	\renewcommand{\arraystretch}{1.25}
	\begin{center}	
		\begin{tabular}{|c|c|c|c|c|c|c|}
			\hline
			Projection & $Y$ & $1/40$ & $1/50$ & $1/100$ & $1/200$ & $1/300$ \\
			\hline
			$\Pi_r^X$ & $L^2(\Omega)$ & 6.5e-02 & 5.8e-02 & 4.2e-02 & 3.1e-02 & 2.6e-02   \\
			$\Pi_r^X$ & $H_0^1(\Omega)$ & 6.8e-02 & 6.2e-02 & 4.8e-02 & 4.0e-02 & 3.6e-02   \\
			$R_r$ & $L^2(\Omega)$ & 6.5e-02 & 5.9e-02 & 4.3e-02 & 3.3e-02 & 2.8e-02   \\
			$R_r$ & $H_0^1(\Omega)$ & 7.0e-02 & 6.4e-02 & 5.1e-02 & 4.4e-02 & 4.1e-02   \\
			\hline
		\end{tabular}
		\caption{Scaling factors for \Cref{thm:POD_DQ1_pointwise_error} as $\Delta t $ changes with $X = H_0^1(\Omega)$ \label{tab:ratios_pointwise_errors_H10}}
	\end{center}
\end{table}

\subsection{ROM Error Bounds}\label{ssec:Ratios_Comp_ROM}
Next we consider the reduced order model error bounds from \Cref{sec:ROM}. Again we use 100 finite element nodes and $r=4$ with varying values of $\Delta t$. First define the following values:
\begin{equation*}
\mathrm{err}_1 = \max_k \|e^k\|^2_{L^2},
\end{equation*}
\begin{equation*}
\mathrm{err}_2 = \max_k \|e^k\|^2_{H_0^1},
\end{equation*}
and 
\begin{equation*}
\mathrm{err}_3 =   \|e^{N}\|_{L^2}^2 +\Delta t \sum_{n=1}^{N-1} \|\nabla e^{n+1}\|_{L^2}^2 .
\end{equation*}
For the first set of scaling factors defined in \Cref{eqn:ratio1}, \Cref{eqn:ratio2}, and \Cref{eqn:ratio3} below we have $X = L^2(\Omega)$. We also compute the scaling factors for the results that use $X = H_0^1(\Omega)$ as the POD space. These scaling factors for the $X = H_0^1(\Omega)$ case are given by \Cref{eqn:ratio4,eqn:ratio5,eqn:ratio6}. For these computations we once again vary the time steps and keep all other parameters constant. The results for both cases of POD basis space can be found in \Cref{tab:ratios_time}.
%The results are shown in \Cref{tab:ratios_L2_time}.

\begin{equation}\label{eqn:ratio1}
C_1 = \mathrm{err}_1  / \left( \sum_{i=r+1}^s \lambda_i^{DQ1} \|\varphi_i - R_r \varphi_i \|_{L^2}^2 + \|\phi_r^1 \|_{L^2}^2 \right) 
\end{equation}

\begin{equation}\label{eqn:ratio2}
C_2 = \mathrm{err}_2 / \left( \sum_{i=r+1}^{s} \lambda_i^{DQ1} \left( \|\varphi_i-R_r\varphi_i\|_{L^2}^2 + \|\nabla (\varphi_i-R_r\varphi_i)\|_{L^2}^2  \right) + \|\phi_r^1\|_{H_0^1}^2 \right)
\end{equation}

\begin{equation}\label{eqn:ratio3}
C_3 = \mathrm{err}_3 / \left( \sum_{i=r+1}^{s} \lambda^{DQ1} (\|\varphi_i-R_r\varphi_i\|_{L^2}^2 + \|\nabla(\varphi_i-R_r\varphi_i)\|_{L^2}^2) + \|\phi_r^1\|_{L^2}^2\right)
\end{equation}

%\begin{table}
%	\begin{center}	
%		\begin{tabular}{|c|c|c|c|c|c|}
%			\hline
%			$\Delta t$ & $1/40$ & $1/50$ & $1/100$ & $1/200$ & $1/300$ \\
%			\hline
%			$C_1$ & 7.3e-03 & 6.4e-03 & 4.4e-03 & 3.3e-03 & 2.9e-03   \\
%			$C_2$ & 2.7e-06 & 4.0e-06 & 9.4e-06 & 1.4e-05 & 1.6e-05   \\			
%			$C_3$ & 8.0e-08 & 9.7e-08 & 1.3e-07 & 1.3e-07 & 1.1e-07   \\
%			\hline
%		\end{tabular}
%		\caption{Scaling factors for the $L^2$ POD basis case as $\Delta t $ changes \label{tab:ratios_L2_time}}
%	\end{center}
%\end{table}

\begin{table}
	\renewcommand{\arraystretch}{1.25}
	\begin{center}	
		\begin{tabular}{|c|c|c|c|c|c|}
			\hline
			$\Delta t$ & $1/40$ & $1/50$ & $1/100$ & $1/200$ & $1/300$ \\
			\hline 
			$C_1$ & 5.8e-02 & 5.1e-02 & 3.7e-02 & 2.8e-02 & 2.5e-02   \\
			$C_2$ & 1.0e-05 & 1.4e-05 & 2.9e-05 & 4.0e-05 & 4.3e-05   \\			
			$C_3$ & 1.9e-06 & 2.4e-06 & 4.1e-05 & 5.2e-06 & 5.3e-06   \\
			$C_4$ & 6.7e-02 & 6.1-e02 & 4.7e-02 & 4.0e-02 & 3.6e-02  \\
			$C_5$ & 6.5e-02 & 5.8e-02 & 4.1e-02 & 3.0e-02 & 2.6e-02   \\
			$C_6$ & 1.1e-02 & 9.3e-03 & 5.7e-03 & 3.8e-03 & 3.1e-03   \\			
			\hline
		\end{tabular}
		\caption{Scaling factors as $\Delta t $ changes for ROM errors \label{tab:ratios_time}}
	\end{center}
\end{table}

%We also compute the scaling factors for the results that use $X = H_0^1$ as the POD space. These scaling factors are given by \Cref{eqn:ratio4,eqn:ratio5,eqn:ratio6} and the computational results can be found in \Cref{tab:ratios_H10_time}.

\begin{equation}\label{eqn:ratio4}
C_4 = \mathrm{err}_1 /\left( \sum_{i=r+1}^s \lambda_i^{DQ1} \|\varphi_i \|_{L^2}^2 + \|\phi_r^1 \|_{L^2}^2 \right)
\end{equation}
\begin{equation}\label{eqn:ratio5}
C_5 = \mathrm{err}_2 /\left( \sum_{i=r+1}^{s} \lambda_i^{DQ1} (\|\varphi_i\|_{L^2}^2 + 1) +\|\phi_r^1\|_{H_0^1}^2 \right)
\end{equation}
\begin{equation}\label{eqn:ratio6}
C_6 = \mathrm{err}_3 /\left( \sum_{i=r+1}^{s} \lambda_i^{DQ1} (1 + \|\varphi_i\|_{L^2}^2) + \|\phi_r^1\|_{L^2}^2 \right)
\end{equation}

%\begin{table}
%	\begin{center}	
%		\begin{tabular}{|c|c|c|c|c|c|}
%			\hline
%			$\Delta t$ & $1/40$ & $1/50$ & $1/100$ & $1/200$ & $1/300$ \\
%			\hline
%			$C_4$ & 2.7e-02 & 2.4e-02 & 1.8e-02 & 1.4e-02 & 2.3e-02   \\
%			$C_5$ & 5.6e-02 & 5.2e-02 & 4.3e-02 & 3.4e-02 & 3.0e-02   \\
%			$C_6$ & 1.0e-02 & 8.9e-03 & 6.2e-03 & 4.5e-03 & 3.8e-03   \\
%			\hline
%		\end{tabular}
%		\caption{Scaling factors for the $H_0^1$ POD basis case as $\Delta t $ changes \label{tab:ratios_H10_time}}
%	\end{center}
%\end{table}

Note that changing the number of finite element nodes and keeping the number of time steps constant yields similar results.   Theoretically we showed the scaling factors should remain bounded, and these computational results support that. For the computations in \Cref{ssec:Ratios_Comp_PW} and \Cref{ssec:Ratios_Comp_ROM}, we note that for larger values of $r$ some of the projection errors, ROM errors, and error formulas become extremely small. In such cases some of the computed scaling factors are very large, but we believe this is caused by round off errors.

%We did observe that some of scaling factors became large for certain cases of large values of $r$ paired with large values of $\Delta t$. However, these specific values of $r$ and $\Delta t$ were much larger than is typically used in applications, and we believe that the scaling factors becoming large for these values is due to round off errors.

%\red{These behave as expected most of the time. Issues arise with the constant when $r$ is larger and the number of timesteps is smaller ($\Delta t$ is larger). In particular, when the number of timesteps is as low as $60$, the constant is increasing by $r=8$ and blows up by $r=10$. }

\section{Conclusions}

In this work, we introduce a new approach to POD using the difference quotients of the snapshot data. Specifically, we derive the POD modes from a data set that includes only the first snapshot and the regular difference quotients. This data set has approximately half the number of snapshots as the standard POD with DQ approach; also, this data set does not include redundant data when the snapshots are linearly independent. For this new approach to POD with DQs, we prove an approximation result for the weighted sum of the POD projection errors of the regular snapshots, and we also prove that we retain all of the numerical analysis benefits of using DQs that were shown in \cite{Koc2020} for the standard POD with DQ approach. Our numerical experiments for a heat equation test problem show that this new approach produces similar reduced order model errors to other known POD approaches.

%We introduce a new computational approach to proper orthogonal decomposition that includes all difference quotients and one data snapshot. We show optimality results with the initial condition taken to be the POD projection of the initial data. New error formulas are shown and preliminary computations suggest that this approach produces similar errors to other known POD approaches. 

This work focuses on numerical analysis results concerning pointwise POD projection errors and the optimality of pointwise in time ROM errors for our new approach to POD with DQs. Further investigation is needed to determine how this new approach compares to other standard POD methods in practical computations. The size of the ROM errors is not considered here and more research is needed to compare the size of the ROM error in this case with those of standard POD and standard DQ POD.

%This work focuses on the optimality of the ROM errors, but further investigation is needed to determine how this new approach will compare to other POD methods in practical computations. The size of the ROM errors is not considered here and more research is needed to compare the size of the ROM error in this case with those of standard POD and standard DQ POD. 

In this work we consider only one PDE and one choice of difference quotient. We focus on the heat equation and leave the Navier-Stokes equations and other more complicated PDEs to be considered elsewhere. Additionally, in this work we consider the difference quotients obtained when using backward Euler for time stepping. Other difference quotients are possible and have been used for snapshot collection in \cite{Herkt2013, JinZhou17, ZhuDedeQuarteroni17}. It is possible that results in this paper can be extended to these other difference quotients, but we leave that to be considered elsewhere. 

\bibliographystyle{plain}
\bibliography{references_POD_approx,pod_differencequotients,POD_theory_new_references,POD_applications}

\begin{thebibliography}{10}

\bibitem{AllaFalconeVolkwein17}
A.~Alla, M.~Falcone, and S.~Volkwein.
\newblock Error analysis for {POD} approximations of infinite horizon problems
  via the dynamic programming approach.
\newblock {\em SIAM J. Control Optim.}, 55(5):3091--3115, 2017.

\bibitem{Baker2012}
C.~G. Baker, K.~A. Gallivan, and P.~Van~Dooren.
\newblock Low-rank incremental methods for computing dominant singular
  subspaces.
\newblock {\em Linear Algebra and its Applications}, 436(8):2866--2888, 2012.

\bibitem{Bergmann05}
Michel Bergmann, Laurent Cordier, and Jean-Pierre Brancher.
\newblock Optimal rotary control of the cylinder wake using proper orthogonal
  decomposition reduced-order model.
\newblock {\em Physics of Fluids}, 17(9):097101, 2005.

\bibitem{Brand2006}
Matthew Brand.
\newblock Fast low-rank modifications of the thin singular value decomposition.
\newblock {\em Linear Algebra and its Applications}, 415(1):20--30, 2006.

\bibitem{Chapelle12}
Dominique Chapelle, Asven Gariah, and Jacques Sainte-Marie.
\newblock Galerkin approximation with proper orthogonal decomposition: new
  error estimates and illustrative examples.
\newblock {\em ESAIM Math. Model. Numer. Anal.}, 46(4):731--757, 2012.

\bibitem{Djouadi08}
S.~M. Djouadi.
\newblock On the optimality of the proper orthogonal decomposition and balanced
  truncation.
\newblock In {\em Proceedings of the 47th IEEE Conference on Decision and
  Control}, pages 4221 --4226, 2008.

\bibitem{Fareed2020}
Hiba Fareed and John~R. Singler.
\newblock Error analysis of an incremental proper orthogonal decomposition
  algorithm for {PDE} simulation data.
\newblock {\em Journal of Computational and Applied Mathematics}, 368:112525,
  14, 2020.

\bibitem{Fareed18}
Hiba Fareed, John~R. Singler, Yangwen Zhang, and Jiguang Shen.
\newblock Incremental proper orthogonal decomposition for {PDE} simulation
  data.
\newblock {\em Computers \& Mathematics with Applications. An International
  Journal}, 75(6):1942--1960, 2018.

\bibitem{GohbergGoldbergKaashoek90}
Israel Gohberg, Seymour Goldberg, and Marinus~A. Kaashoek.
\newblock {\em {C}lasses of {L}inear {O}perators. {V}ol. {I}}, volume~49 of
  {\em Operator Theory: Advances and Applications}.
\newblock Birkh\"auser Verlag, Basel, 1990.

\bibitem{Graessle2019}
Carmen Gr\"{a}{\ss}le, Michael Hinze, Jens Lang, and Sebastian Ullmann.
\newblock P{OD} model order reduction with space-adapted snapshots for
  incompressible flows.
\newblock {\em Advances in Computational Mathematics}, 45(5-6):2401--2428,
  2019.

\bibitem{Gu2021}
Haotian Gu, Jack Xin, and Zhiwen Zhang.
\newblock Error {E}stimates for a {POD} {M}ethod for {S}olving {V}iscous
  {G}-{E}quations in {I}ncompressible {C}ellular {F}lows.
\newblock {\em SIAM Journal on Scientific Computing}, 43(1):A636--A662, 2021.

\bibitem{GubischVolkwein17}
Martin Gubisch and Stefan Volkwein.
\newblock Proper orthogonal decomposition for linear-quadratic optimal control.
\newblock In {\em Model reduction and approximation}, volume~15 of {\em Comput.
  Sci. Eng.}, pages 3--63. SIAM, Philadelphia, PA, 2017.

\bibitem{Herkt2013}
Sabrina Herkt, Michael Hinze, and Rene Pinnau.
\newblock Convergence analysis of {G}alerkin {POD} for linear second order
  evolution equations.
\newblock {\em Electronic Transactions on Numerical Analysis}, 40:321--337,
  2013.

\bibitem{Higham2020}
J.~E. Higham, M.~Shahnam, and A.~Vaidheeswaran.
\newblock Using a proper orthogonal decomposition to elucidate features in
  granular flows.
\newblock {\em Granular Matter}, 22(4), 2020.

\bibitem{Hijazi2020}
Saddam Hijazi, Giovanni Stabile, Andrea Mola, and Gianluigi Rozza.
\newblock Data-driven {POD}-{G}alerkin reduced order model for turbulent flows.
\newblock {\em Journal of Computational Physics}, 416:109513, 30, 2020.

\bibitem{Himpe2018}
Christian Himpe, Tobias Leibner, and Stephan Rave.
\newblock Hierarchical approximate proper orthogonal decomposition.
\newblock {\em SIAM Journal on Scientific Computing}, 40(5):A3267--A3292, 2018.

\bibitem{HinzeVolkwein}
Michael Hinze and Stefan Volkwein.
\newblock Proper orthogonal decomposition surrogate models for nonlinear
  dynamical systems: Error estimates and suboptimal control.
\newblock In {\em Lecture Notes in Computational Science and Engineering},
  pages 261--306. Springer-Verlag, 2005.

\bibitem{HolmesLumleyBerkoozRowley12}
Philip Holmes, John~L. Lumley, Gahl Berkooz, and Clarence~W. Rowley.
\newblock {\em Turbulence, coherent structures, dynamical systems and
  symmetry}.
\newblock Cambridge Monographs on Mechanics. Cambridge University Press,
  Cambridge, second edition, 2012.

\bibitem{Hoemberg2003}
D.~H\"{o}mberg and S.~Volkwein.
\newblock Control of laser surface hardening by a reduced-order approach using
  proper orthogonal decomposition.
\newblock {\em Mathematical and Computer Modelling}, 38(10):1003--1028, 2003.

\bibitem{IliescuWang13}
Traian Iliescu and Zhu Wang.
\newblock Variational multiscale proper orthogonal decomposition:
  convection-dominated convection-diffusion-reaction equations.
\newblock {\em Math. Comp.}, 82(283):1357--1378, 2013.

\bibitem{IliescuWang14_DQ}
Traian Iliescu and Zhu Wang.
\newblock Are the snapshot difference quotients needed in the proper orthogonal
  decomposition?
\newblock {\em SIAM Journal on Scientific Computing}, 36(3):A1221--A1250, 2014.

\bibitem{IliescuWang14}
Traian Iliescu and Zhu Wang.
\newblock Variational multiscale proper orthogonal decomposition:
  {N}avier-{S}tokes equations.
\newblock {\em Numer. Methods Partial Differential Equations}, 30(2):641--663,
  2014.

\bibitem{JinZhou17}
Bangti Jin and Zhi Zhou.
\newblock An analysis of {G}alerkin proper orthogonal decomposition for
  subdiffusion.
\newblock {\em ESAIM. Mathematical Modelling and Numerical Analysis},
  51(1):89--113, 2017.

\bibitem{Karasoezen2018}
B\"{u}lent Karas\"{o}zen and Murat Uzunca.
\newblock Energy preserving model order reduction of the nonlinear
  {S}chr\"{o}dinger equation.
\newblock {\em Advances in Computational Mathematics}, 44(6):1769--1796, 2018.

\bibitem{Kato95}
Tosio Kato.
\newblock {\em Perturbation theory for linear operators}.
\newblock Classics in Mathematics. Springer-Verlag, Berlin, 1995.
\newblock Reprint of the 1980 edition.

\bibitem{KeanSchneier19}
Kiera Kean and Michael Schneier.
\newblock Error analysis of supremizer pressure recovery for {POD} based
  reduced-order models of the time-dependent {N}avier-{S}tokes equations.
\newblock {\em SIAM Journal on Numerical Analysis}, 58(4):2235--2264, 2020.

\bibitem{Koc2019}
Birgul Koc, Muhammad Mohebujjaman, Changhong Mou, and Traian Iliescu.
\newblock Commutation error in reduced order modeling of fluid flows.
\newblock {\em Advances in Computational Mathematics}, 45(5-6):2587--2621,
  2019.

\bibitem{Koc2020}
Birgul Koc, Samuele Rubino, Michael Schneier, John~R. Singler, and Traian
  Iliescu.
\newblock On optimal pointwise in time error bounds and difference quotients
  for the proper orthogonal decomposition.
\newblock {\em SIAM Journal on Numerical Analysis}.
\newblock To appear, arXiv:2010.03750.

\bibitem{Kostova-VassilevskaOxberry18}
Tanya Kostova-Vassilevska and Geoffrey~M. Oxberry.
\newblock Model reduction of dynamical systems by proper orthogonal
  decomposition: error bounds and comparison of methods using snapshots from
  the solution and the time derivatives.
\newblock {\em J. Comput. Appl. Math.}, 330:553--573, 2018.

\bibitem{Kunisch2001}
K.~Kunisch and S.~Volkwein.
\newblock Galerkin proper orthogonal decomposition methods for parabolic
  problems.
\newblock {\em Numerische Mathematik}, 90(1):117--148, 2001.

\bibitem{KunischVolkwein02}
K.~Kunisch and S.~Volkwein.
\newblock Galerkin proper orthogonal decomposition methods for a general
  equation in fluid dynamics.
\newblock {\em SIAM J. Numer. Anal.}, 40(2):492--515, 2002.

\bibitem{Kunisch08}
Karl Kunisch and Stefan Volkwein.
\newblock Proper orthogonal decomposition for optimality systems.
\newblock {\em M2AN Math. Model. Numer. Anal.}, 42(1):1--23, 2008.

\bibitem{Lax02}
Peter~D. Lax.
\newblock {\em Functional analysis}.
\newblock Pure and Applied Mathematics (New York). Wiley-Interscience [John
  Wiley \& Sons], New York, 2002.

\bibitem{Leibfritz2007}
Friedemann Leibfritz and Stefan Volkwein.
\newblock Numerical feedback controller design for {PDE} systems using model
  reduction: techniques and case studies.
\newblock In {\em Real-time {PDE}-constrained optimization}, volume~3 of {\em
  Comput. Sci. Eng.}, pages 53--72. SIAM, Philadelphia, PA, 2007.

\bibitem{Liang02}
Y.~C. Liang, H.~P. Lee, S.~P. Lim, W.~Z. Lin, K.~H. Lee, and C.~G. Wu.
\newblock Proper orthogonal decomposition and its applications. {I}. {T}heory.
\newblock {\em J. Sound Vibration}, 252(3):527--544, 2002.

\bibitem{LockeSingler20}
Sarah Locke and John Singler.
\newblock New proper orthogonal decomposition approximation theory for {PDE}
  solution data.
\newblock {\em SIAM Journal on Numerical Analysis}, 58(6):3251--3285, 2020.

\bibitem{LuoChenNavonYang08}
Zhendong Luo, Jing Chen, I.~M. Navon, and Xiaozhong Yang.
\newblock Mixed finite element formulation and error estimates based on proper
  orthogonal decomposition for the nonstationary {N}avier-{S}tokes equations.
\newblock {\em SIAM J. Numer. Anal.}, 47(1):1--19, 2008/09.

\bibitem{Nguyen17}
Van~Bo Nguyen, H.-S. Dou, K.~Willcox, and Boo-Cheong Khoo.
\newblock Model order reduction for reacting flows: Laminar {G}aussian flame
  applications.
\newblock In {\em 30th International Symposium on Shock Waves 1}, pages
  337--343. Springer International Publishing, 2017.

\bibitem{QuarteroniManzoniNegri16}
Alfio Quarteroni, Andrea Manzoni, and Federico Negri.
\newblock {\em Reduced basis methods for partial differential equations},
  volume~92 of {\em Unitext}.
\newblock Springer, Cham, 2016.

\bibitem{Rathinam2003}
Muruhan Rathinam and Linda~R. Petzold.
\newblock A new look at proper orthogonal decomposition.
\newblock {\em SIAM Journal on Numerical Analysis}, 41(5):1893--1925, 2003.

\bibitem{ReedSimon80}
Michael Reed and Barry Simon.
\newblock {\em Methods of modern mathematical physics {I}: {F}unctional
  analysis}.
\newblock Academic Press, Inc., New York, second edition, 1980.

\bibitem{ROWLEY2005}
C.~W. Rowley.
\newblock Model reduction for fluids, using balanced proper orthogonal
  decomposition.
\newblock {\em Internat. J. Bifur. Chaos Appl. Sci. Engrg.}, 15(3):997--1013,
  2005.

\bibitem{Rubino18}
Samuele Rubino.
\newblock A streamline derivative {POD}-{ROM} for advection-diffusion-reaction
  equations.
\newblock {\em {ESAIM}: Proceedings and Surveys}, 64:121--136, 2018.

\bibitem{Sachs2013}
Ekkehard~W. Sachs and Matthias Schu.
\newblock {\it {A} priori} error estimates for reduced order models in finance.
\newblock {\em ESAIM. Mathematical Modelling and Numerical Analysis},
  47(2):449--469, 2013.

\bibitem{Singler14}
John~R. Singler.
\newblock New {POD} error expressions, error bounds, and asymptotic results for
  reduced order models of parabolic {PDE}s.
\newblock {\em SIAM J. Numer. Anal.}, 52(2):852--876, 2014.

\bibitem{Volkwein04}
S.~Volkwein.
\newblock Interpretation of proper orthogonal decomposition as singular value
  decomposition and {HJB}-based feedback design.
\newblock In {\em Proceedings of the Sixteenth International Symposium on
  Mathematical Theory of Networks and Systems (MTNS)}, 2004.

\bibitem{Willcox02}
K.~Willcox and J.~Peraire.
\newblock Balanced model reduction via the proper orthogonal decomposition.
\newblock {\em {AIAA} Journal}, 40(11):2323--2330, 2002.

\bibitem{ZhuDedeQuarteroni17}
Shengfeng Zhu, Luca Ded\`e, and Alfio Quarteroni.
\newblock Isogeometric analysis and proper orthogonal decomposition for the
  acoustic wave equation.
\newblock {\em ESAIM. Mathematical Modelling and Numerical Analysis},
  51(4):1197--1221, 2017.

\end{thebibliography}

\end{document}